\documentclass[12pt]{amsart}
\usepackage{amsmath,amssymb,latexsym,esint,cite,mathrsfs}
\usepackage[pagewise]{lineno}
\usepackage{verbatim,wasysym}
\usepackage[left=2.6cm,right=2.6cm,top=2.9cm,bottom=2.9cm]{geometry}
\usepackage{xcolor}
\usepackage{amsmath}
\usepackage{}
\usepackage{color}
\usepackage{bm}
\usepackage{amsfonts}
\usepackage{a4wide}
\usepackage{amsmath,amsthm,amssymb,amscd}
\usepackage{latexsym}
\usepackage{hyperref}
\usepackage[numbers,sort&compress]{natbib}
\usepackage{hypernat}
\allowdisplaybreaks
\numberwithin{equation}{section}

\newtheorem{Thm}{Theorem}[section]
\newtheorem{Lem}[Thm]{Lemma}

\newtheorem{Prop}[Thm]{Proposition}

\newtheorem{Rem}[Thm]{Remark}

\def\R{\mathbb{R}}

\begin{document}

\title[ Normalized solutions for a Schr\"odinger-Newton system ]
{Existence and local uniqueness of normalized peak solutions for a Schr\"odinger-Newton system}

\author{Qing Guo,\,\,Peng Luo,\,\,Chunhua Wang$^\dagger$ and Jing Yang}

 \address[Qing Guo]{College of Science, Minzu University of China, Beijing 100081, China}

\email{guoqing0117@163.com}

 \address[Peng Luo]{School of Mathematics and Statistics and Hubei Key Laboratory of Mathematical Sciences, Central China Normal University, Wuhan 430079, China}

\email{pluo@mail.ccnu.edu.cn}

\address[Chunhua Wang]{School of Mathematics and Statistics and Hubei Key Laboratory of Mathematical Sciences, Central China Normal University, Wuhan 430079, China}

\email{chunhuawang@mail.ccnu.edu.cn}

\address[Jing Yang]{School of Science, Jiangsu University of Science and Technology, Zhenjiang 212003, China}

\email{yyangecho@163.com}

\thanks{$^\dagger$ Corresponding author: Chunhua Wang}

\begin{abstract}
 In this paper, we investigate the
existence and local uniqueness of normalized peak solutions for a Schr\"odinger-Newton system
under the assumption that the trapping potential is degenerate and has non-isolated critical points.

First we investigate the
existence and local uniqueness of normalized single-peak solutions for the Schr\"odinger-Newton system.
Precisely, we   give  the precise description of the chemical potential $\mu$ and
the attractive interaction $a$. Then we apply the finite dimensional reduction method
to obtain the existence of single-peak solutions. Furthermore, using various local Pohozaev identities, blow-up analysis and the maximum principle,
  we prove the local uniqueness of single-peak solutions by precise analysis of the concentrated points and the Lagrange multiplier.
Finally, we also prove the nonexistence of multi-peak solutions for the Schr\"odinger-Newton system,
which is markedly different from the corresponding Schr\"odinger equation. The nonlocal term
results in this difference.

 The main difficulties come from the
  estimates on Lagrange multiplier, the different degenerate rates along different directions at the critical point of $P(x)$
  and some complicated estimates involved by the nonlocal term.
   To our best knowledge, it may be the first time to study the existence and
  local uniqueness of solutions with
 prescribed $L^{2}$-norm for the Schr\"odinger-Newton system.


\end{abstract}

\maketitle

\section{Introduction and our main results}
In this paper, we investigate the following nonlinear Schr\"odinger-Newton system
\begin{equation}\label{main-eq}
 \begin{cases}
 -\Delta u+P(x)u=a\psi u+\mu u,\,&x\in \R^{3},\\
-\Delta \psi =\frac{u^2}{2},\,& x\in \R^{3},
 \end{cases}
 \end{equation}
under the mass constraint $\int_{\R^{3}}u^{2}(x)dx=1,$ where $P(x)$
is a degenerate trapping potential with non-isolated critical points,
 $\mu\in\R$  and $a\geq 0 $ denote the chemical potential and the
attractive interaction respectively.

The Schr\"odinger-Newton problem arises from describing the quantum mechanics of a polaron at rest, see \cite{Penrose}.
 Also it was used to describe an electron trapped in its own hole in a certain approximating to
 Hartree-Fock theory of one component plasma in \cite{Lieb1}.
  In addition, Penrose in \cite{Penrose} applied it as a model of self-gravitating matter,
  where quantum state reduction is understood as a gravitational phenomenon. Specifically, if $m$ is the mass of the point, the interaction leads to the system in $\R^3$
 \begin{equation}\label{ll}
 \begin{cases}
- \frac{\epsilon^2}{2m}\Delta u+P(x)u=\psi u,\,&x\in \R^{3},\\
 -\Delta \psi =4\pi \tau u^2,\,&x\in \R^{3},
 \end{cases}
 \end{equation}
 where $u$ is the wave function, $\psi$ is the gravitational potential energy, $P(x)$ is
 a given Schr\"odinger potential, $\epsilon$ is the Planck constant, $\tau=Gm^2$ and $G$ is
 the Newton's constant of gravitation.

 Set
 \begin{equation*}
 u(x)\mapsto\frac{  u}{4\epsilon \sqrt{\pi \tau m}}, ~~P(x)\mapsto\frac{1}{2m}  P(x),~~\psi(x)\mapsto\frac{1}{2m}  \psi(x).
 \end{equation*}
 Then system \eqref{ll} can be written, maintaining the original notations, as
  \begin{equation}\label{ll1}
 \begin{cases}
 -\epsilon^2\Delta u+P(x)u=\psi u,\,&x\in \R^{3},\\
 -\epsilon^2\Delta \psi =\frac{u^2}{2},\,& x\in \R^{3}.
 \end{cases}
 \end{equation}
  From the second equation of \eqref{ll1}, we know
 $ \psi(x)=\frac{1}{8\pi \epsilon^2}
 \big(\int_{\R^3}\frac{u^2(y)}{|x-y|}dy\big).$
 Then the system \eqref{ll1} turns into the following  single nonlocal equation
 \begin{flalign}\label{1.2}
 -\epsilon^2\Delta u+P(x)u=\frac{1}{8\pi \epsilon^2}
 \big(\int_{\R^3}\frac{u^2(y)}{|x-y|}dy\big)u,~x\in \R^3,
 \end{flalign}
which also appears in the study of standing waves for the following nonlinear Hartree equations
\begin{flalign*}
i\epsilon\frac{\partial \varphi}{\partial t}=-\epsilon^2\Delta_x\varphi +(P(x)+E)\varphi-\frac{1}{8\pi \epsilon^2}
 \big(\int_{\R^3}\frac{\varphi^2(y)}{|x-y|}dy\big)\varphi,~(x,t)\in \R^3\times\R^+,
\end{flalign*}
with the form $\varphi(x,t)=e^{-iEt/\epsilon}u(x),$ where $i$ is the imaginary unit and $\epsilon$ is the Planck constant.

In recent decades,  problem \eqref{1.2} has been extensively investigated.
When $\epsilon=1$ and $P(x)=1$, \eqref{1.2} becomes
 \begin{flalign}\label{l1.2}
 - \Delta u+ u=\frac{1}{8\pi}
 \big(\int_{\R^3}\frac{u^2(y)}{|x-y|}dy\big)u,~x\in \R^3.
 \end{flalign}
The existence and uniqueness of ground states for \eqref{l1.2} was obtained with variational methods by Lieb \cite{Lieb1}, Lions \cite{Lions} and Menzala \cite{Menzala}.
The nondegeneracy of the ground states for \eqref{l1.2} was proved by  Tod-Moroz \cite{Tod} and Wei-Winter \cite{Wei}.

When $\epsilon$ is small and  $P(x)$ is not a constant, in \cite{Lions1} Lions
proved the existence of solutions with ground states for \eqref{1.2}
  under some conditions on $P(x)$
 since problem \eqref{1.2} has a variational structure.
 Moreover,  the solution with  ground states concentrates at certain  point.
Applying the finite dimensional reduction method, in \cite{Wei}
Wei and Winter  proved that \eqref{1.2} has a solution concentrating at $k$ points which are the local minimum points of $P(x).$
 This also means the existence of multiple solutions. Very recently,
 in \cite{LPW-20CV} Luo, Peng and Wang proved the uniqueness of the concentrated solutions of \eqref{1.2} obtained in \cite{Wei}
 by some local Pohozaev identities, blow-up analysis and the maximum principle.
  For more other results of the existence of solutions with concentration, one can refer to  \cite{Cingolani, Secchi,Vaira,GW-20ANS} and the references therein. For a very similar nonlocal problem i.e. Schr\"odinger-Possion equations,
 one can refer to \cite{LPW-11JMP,Li,WY-16DCDS}, where some existence of peak solutions are obtained
 by the finite dimensional reduction method
under various assumptions of the potential function.

Actually, the Schr\"odinger-Newton problem \eqref{1.2} is a special type of following Choquard equation:
 \begin{flalign}\label{ll2-1223}
 -\epsilon^2\Delta u+P(x)u=
 \frac{a}{\epsilon^{2}}(I_{\alpha}\ast F(u))f(u)+\mu u,\,\,N\geq 2,\,\,p>1,
 ,~x\in \R^N,
 \end{flalign}
 where $F\in C^{1}(\R,\R),F'(s)=f(s),\mu\in\R$ denotes the chemical potential,
 $a\geq 0$ denotes the attractive interaction  and
 the Riesz potential $I_{\alpha}:\R^{N}\rightarrow \R$ is defined(cf. Ref \cite{R-1949}) as
 \begin{flalign}\label{eq-I}
 I_{\alpha}:=\frac{\Gamma(\frac{N-\alpha}{2})}{\Gamma(\frac{\alpha}{2})\pi^{\frac{N}{2}}2^{\alpha}}\frac{1}{|x|^{N-\alpha}},
 \,\,x\in \R^{N}\backslash\{0\},\,\,\,\alpha\in(0,N).
 \end{flalign}
Especially, when $N=3$ and $\alpha=2$ in \eqref{eq-I}, then $I_{\alpha}=\frac{1}{4\pi|x|},\,\, x\in \R^{3}\backslash\{0\}.$

There are many results about the case that $f(s)=|s|^{p-2}s.$
When $\epsilon=1,N=3,\alpha=2, P(x)=constant>0,\mu=0$
and $p\rightarrow 2,$
in \cite{Xiang} Xiang proved the uniqueness and non-degeneracy of ground
states to \eqref{ll2-1223}.
When $\epsilon=1,N=3,4,5,P(x)=1,\alpha=2,p=2$ and $\mu=0$,
in \cite{Chen-16}
Chen proved
the nondegeneracy of ground states of \eqref{ll2-1223}.   As an application of the  non-degeneracy
result he obtained,
 he then used a Lyapunov-Schmidt reduction argument to construct multiple semi-classical solutions to \eqref{ll2} with an external potential.
When $N=3,\alpha=2$ and $\mu=0$ in \eqref{ll2-1223}, in \cite{Moroz,Moroz1}
Moroz and Van Schaftingen obtained some results about the existence and concentration of positive solutions to the Choquard equation.
For ground states of \eqref{ll2-1223} with $\epsilon=1,$
one can refer to \cite{DGL-18JMP,GLW-18JMP,LZ-19ZAMP,LXZ-16-JMP,Ye-16TMNA,LY-14JMP} under various assumptions of trapping potential $P(x),$
which can be described equivalently by positive $L^{2}$  minimizers of
the following Hartree-type energy functional(c.f. \cite{GLW-18JMP})
$$
E_{a}(u)=\int_{\R^{N}}(|\nabla u|^{2}+P(x)u^{2})dx
-\frac{a}{p}\int_{\R^{N}}(I_{\alpha}\ast |u(x)|^{p})|u(x)|^{p}dx.
$$

Very recently, taking $a=1,$ in \cite{CT-2019} Cingolani and Tanaka developed a new variational approach
and proved the existence of a family of solutions concentrating to a local minimum of $P(x)$
as $\epsilon\rightarrow 0$ under general conditions on $F(s).$
In \cite{CGT}, assuming that the nonlinear term is subcritical and satisfies
 almost optimal assumptions, Cingolani1, Gallo and Tanaka proved the existence of a spherically symmetric solution to
the following fractional Schr\"{o}dinger equation with a nonlocal nonlinearity of Choquard type
$$
 \begin{cases}
 (-\Delta)^{s} u+\mu u=(I_{\alpha}\ast F(u))f(u),\,&x\in \R^{N},\vspace{0.1cm}\\
\int_{\R^{N}}u^{2}=c,\,\,&(\mu,u)\in (0,+\infty)\times H_{r}^{s}(\R^{N)},
 \end{cases}
 $$
 where $s\in (0,1), N\geq 2, \alpha\in (0,N),F\in C^{1}(\R,\R), f(s)=F'(s)$ and $c>0.$
 Also they proved the solution obtained is a ground state.
 In \cite{JL}, Jeanjean and Le  studied the existence of solutions to the Schr\"{o}dinger-Poisson-Slater equation
  \begin{equation}\label{1223-1}
 -\Delta u+\mu u-\gamma(|x|^{-1}\ast|u|^{2})u-\beta|u|^{p-2}u,\,x\in \R^{3},\,\,\int_{\R^{N}}u^{2}=c,\,\,
 u\in H^{1}(\R^{3}),
 \end{equation}
 where $c>0,\gamma\in \R,p\in (\frac{10}{3},6]$ and $\beta\in \R.$
 When $\gamma>0$ and $\beta>0$ and $p\in (\frac{10}{3}, 6],$
 they proved that there exists a $c_{1}>0$ such that, for any $c\in(0, c_{1})$,
 \eqref{1223-1} admits two solutions $u_{+c}$ and $u_{-c}$ which can be characterized respectively as a local minima
 and as a mountain pass critical point of the associated Energy functional restricted to the norm constraint.
 In the case $\gamma>0 $ and $\beta<0,$ they proved that, for any $p\in (\frac{10}{3}, 6]$
 and any $c>0,$  \eqref{1223-1} admits a solution which is a global minimizer.
  When $ \gamma<0, \beta>0$ and $ p=6,$ they proved that \eqref{1223-1} does not admit positive solutions.

 Note that all the references about the existence of normalized solutions above are
 studied by the variational methods. Very recently,  Luo etc in \cite{LPY-19} applied the finite-dimensional
 reduction method to study the existence of normalized solutions for the Bose-Einstein condensates.
Also in \cite{PPVV-2021}, Pellacci etc studied normalized solutions for a nonlinear elliptic system
by the Lyapunov-Schmidt reduction.

 To our best knowledge, up to now there are no results about the existence and the local uniqueness of peak solutions
to \eqref{8-18-1}, that is,
\begin{equation}\label{8-18-1}
-\Delta u +P(x) u = \frac{a}{8\pi} \int_{\R^{3}}\frac{u^{2}(y)}{|x-y|}dy\,u +\mu u,\quad \text{in}\;\mathbb R^3,
\end{equation}
where $P(x)$ denotes a class of degenerate trapping potential which has non-isolated critical points.
So in this paper, we are aimed to investigate these problems. Precisely, we
study the existence and the local uniqueness of peak solutions for \eqref{main-eq}  with
  the following $L^{2}$ constraint
\begin{equation}\label{8-28-3}
\int_{\mathbb R^3} u^2=1.
\end{equation}

We first recall the existence result for the ground state.
 Denote by  $U(x)$ the unique positive solution of
\begin{equation}\label{big-u}
-\Delta u+u=\int_{\R^{3}}\frac{u^{2}(y)}{|x-y|}dy\,u(x), ~u\in H^1(\R^3).
\end{equation}

Form \cite{Lieb,Wei}, we know that $U(0)=\displaystyle\max_{x\in \R^3}U(x)$
and
the solution $U(x)$ is strictly decreasing and
\begin{flalign*}
\lim_{|x|\rightarrow \infty}U(x)e^{|x|}|x|=\lambda_0>0,~
\lim_{|x|\rightarrow \infty}\frac{U'(x)}{U(x)}=-1,
\end{flalign*}
for some constant $\lambda_0>0$. Moreover, if $\phi(x)\in H^1(\R^3)$ solves the linearized equation
\begin{flalign*}
-\Delta \phi(x) +\phi(x)=\displaystyle\frac{1}{8\pi}
 \big(\int_{\R^3}\frac{U^2(y)}{|x-y|}dy\big)\phi(x)+
 \frac{1}{4\pi}
 \big(\int_{\R^3}\frac{U(y)\phi(y)}{|x-y|}dy\big)U(x),
\end{flalign*}
then $\phi(x)$ is a linear combination of $\frac{\partial U}{\partial x_j},j=1,2,3.$

  In this paper, we let
$a_*=\displaystyle\int_{\R^3}U^2$.
First of all, we study the following problem without constraint:
\begin{equation}\label{1-23-5}
\begin{cases}
-\Delta w + (\lambda + P(x) ) w =\int _{\R^{3}}\frac{w^{2}(y)}{|x-y|}dy\,w, & \text{in}\; \mathbb R^3;\vspace{2mm}\\
w\in H^1(\mathbb R^3),
\end{cases}
\end{equation}
where $\lambda>0$ is a large parameter.  It is well known that for large $\lambda>0$, we can construct various positive solutions concentrating at
some stable critical points of $P(x).$  Particularly, we can construct positive single-peak solutions for \eqref{1-23-5} in the sense that
\begin{equation*}\label{2-23-5}
w_\lambda (x)=  \lambda  \Bigl( U\bigl( \sqrt\lambda(x-x_{\lambda})\bigr)+\varpi_\lambda (x)\Bigr),
\end{equation*}
with  $\displaystyle\int_{\R^3}\big[\frac{1}{\lambda}|\nabla \varpi_\lambda|^2+  \varpi^2_\lambda\big]=o\big(\lambda^{-\frac{3}{2}}\big)$.
Let $
u_\lambda= \frac { w_\lambda} {\bigl(\int_{\mathbb R^3} w_\lambda^2\bigr)^{1/2}}.$
Then $\displaystyle\int_{\mathbb R^3} u_\lambda^2=1,$ and
\begin{equation*}\label{3-23-5}
\begin{cases}
-\Delta u_\lambda + (\lambda + P(x) ) u_\lambda = a_\lambda \int_{\R^{3}}\frac{u_\lambda^2(y)}{|x-y|}dy\,u_\lambda(x), & \text{in}\; \mathbb R^3;\vspace{2mm}\\
u_\lambda\in H^1(\mathbb R^3),
\end{cases}
\end{equation*}
with
\[
a_\lambda= \int_{\mathbb R^3} w_\lambda^2 =\lambda^{2-\frac 32} \Bigl( \int_{\mathbb R^3} U^2 +o(1)\Bigr)
=\lambda^{\frac{1}{2}} \bigl(a_* +o(1)\bigr).
\]
Note that $a_\lambda>0,$ and as $\lambda \to +\infty,$ so $a_\lambda\to +\infty.$  Therefore, we obtain
a concentrated solution with single peak for \eqref{8-18-1}--\eqref{8-28-3} with $\mu=-\lambda$ and suitable $a_\lambda$.  Now the crucial
question is for any $a>0$ large, whether we can choose a suitable large $\lambda_a>0$, such that
\eqref{8-18-1}--\eqref{8-28-3} holds with
\begin{equation*}\label{10-23-5}
\mu=-\lambda_a,\quad u_a= \frac { w_{\lambda_a}} {\bigl(\displaystyle\int_{\mathbb R^3} w_{\lambda_a}^2\bigr)^{1/2}}.
\end{equation*}

The above discussions show that the existence of concentrated solutions for \eqref{8-18-1}--\eqref{8-28-3} is closely related to the existence of peaked solutions
for the nonlinear Schr\"odinger equation \eqref{1-23-5}.
In this paper, we will mainly investigate
concentrated  solutions $u_a$  of  \eqref{8-18-1}--\eqref{8-28-3} in the  sense that
$$
\max_{x\in B_\vartheta(b_{0})} u_{a}(x) \to +\infty,~\mbox{while}~u_a(x)\to 0~\mbox{uniformly in}~\mathbb R^3\setminus   B_\vartheta(b_{0}), ~\mbox{for any}~\vartheta>0,
$$
as $a \to +\infty,$ where $b_{0}$ is a point in $\mathbb R^3.$
Here, we are concerned with the following three aspects:
 possible values for $\mu_a;$ the exact location of  the concentrated points of the concentrated solutions
 for \eqref{8-18-1}--\eqref{8-28-3};
  the existence and the local uniqueness of the concentrated solutions for \eqref{8-18-1}--\eqref{8-28-3}.
 \smallskip

 The first result of our paper is
as follows.

\begin{Thm}\label{th1-24-5}
Suppose that $u_a$ is a solution of \eqref{8-18-1}--\eqref{8-28-3} concentrated at some points
as $a\to +\infty.$  Then it holds
\begin{equation*}
\mu_a\to -\infty,~\mbox{as}~a\to +\infty.
\end{equation*}
 Moreover, $u_a$ satisfies
\begin{equation}\label{20-23-5-1223}
u_a (x)=  \frac{-\mu_a}{\sqrt{a}} \Bigl( U\bigl( \sqrt {-\mu_a}(x-x_{a})\bigr)+\varpi_a(x)\Bigr),
\end{equation}
with
$\displaystyle\int_{\R^3}\big[-\frac{1}{\mu_a}|\nabla \varpi_a|^2+\varpi_a^2\big]=o\big(\frac{1}{(\sqrt{-\mu_a})^{3}}\big).$

\end{Thm}

Throughout this paper, we call $u_a$  a single-peak solution of \eqref{8-18-1}--\eqref{8-28-3}
if $u_a$ satisfies \eqref{20-23-5-1223}.
 To the best knowledge of us,  if the critical point of $P(x)$ is not isolated, not much is known for the exact location of the concentrated point, nor for the local uniqueness
of the solutions.
In the paper, we assume that $P(x)$ obtains its local minimum or local maximum at
 $\Gamma_i$ ($i=1,\cdots,m$) and $\Gamma_i$ is a closed $2$ dimensional hyper-surface satisfying
 $\Gamma_i \bigcap \Gamma_j=\emptyset$ for $i\neq j$.
  More precisely, we assume that the following conditions hold.

\medskip

\noindent\emph{\textup{\textbf{($P$).}} There exist $\delta>0$ and some $C^2$ compact
 hypersurfaces $\Gamma_i~(i=1,\cdots,m)$ without boundary, satisfying
$$
P(x)=P_i, \;\;  \frac{\partial P(x)}{\partial \nu_i}=0,\;\; \frac{\partial^2 P(x)}{\partial \nu_i^2}\ne 0, ~\mbox{for any}~x \in \Gamma_i~~\mbox{and}~i=1,\cdots,m,
$$
where $P_i\in \R$, $\nu_i$ is the unit outward normal of
$\Gamma_i$  at $x\in \Gamma_i.$
   Moreover, $P(x) \in C^4\big(\bigcup^m_{i=1}W_{\delta,i}\big)$
   and $P(x)=O\big(e^{\alpha|x|})$ for some $\alpha \in (0, 2).$
Here we denote  $W_{\delta,i}:=\{x\in \R^3, dist(x,\Gamma_i)<\delta\}.$
}

\smallskip

We would like to point out that
the assumption $(P)$ was first introduced in \cite{LPY-19}
by Luo etc
A specific example of $P(x)$ was also given by \cite{LPY-19}.
Observe that the assumption $(P)$ implies that $P(x)$ obtains its local minimum or local maximum on the hypersurface
$\Gamma_i$ for $i=1,\cdots,m.$ It is also easy
 to see that if $\delta>0$ is small, the set $\Gamma_{t,i}=\bigl\{  x:  P(x)= t\bigr\}\bigcap  W_{\delta,i}$ consists of two
 compact hypersurfaces
in $\mathbb R^3$  without boundary for $t\in [P_i, P_i+\theta]$
( or $ t\in [P_i-\theta, P_i]$) provided $\theta>0$ is small.
Moreover,   the outward unit normal vector $\nu_{t,i}(x)$ and
the $j$-th principal tangential unit vector $\tau_{t,i,j}(x)$($j=1, 2$) of  $\Gamma_{t,i}$  at
$x$ are Lip-continuous in $W_{\delta,i}.$

\begin{Rem}
When $m=1,$ for simplicity of notations we denote $P_{0}=:P_{i}$ and omit all the subscript $i.$
\end{Rem}

Applying the local Pohozaev identities, we can easily prove that a single
peak solution of \eqref{8-18-1}--\eqref{8-28-3}
must concentrate at a critical point of $P(x),$   for which we can also refer to
\cite{Grossi}.
If the assumption $(P)$ holds
and the concentrated points belong to $\Gamma,$
we ask further where the concentrating points locate on $\Gamma.$
The following result gives our answer for this question.

\begin{Thm}\label{nth1.1}
Assume that the assumption \textup{($P$)} holds.
 If $u_a$ is a
single-peak solution of \eqref{8-18-1}--\eqref{8-28-3}, concentrating at $\{b_0\}$ with $b_0\in \Gamma$
as $a$ goes to $\infty,$ then
\begin{equation}\label{1.6-1223}
(D_{\tau_{j}} \Delta  V)(b_0)=0,~\mbox{with}~j=1,2.
\end{equation}
where $\tau_{j}$ is the $j$-th principal tangential unit vector of $\Gamma$ at $b_0.$
\end{Thm}

Theorem~\ref{nth1.1} implies
that not every $\{b_0\}$ with $b_0\in\Gamma$ can generate a single-peak solution for \eqref{8-18-1}--\eqref{8-28-3}.
 In order to investigate the converse of Theorem~\ref{nth1.1}, we
 have to add the following non-degenerate assumption on the critical
 point of $P(x).$
   We say that $x_0 \in \Gamma$ is non-degenerate on $\Gamma$ if it satisfies:
\begin{equation*}
\frac{\partial^2P(x_0)}{\partial \nu ^2}\neq 0~~\mbox{and}~~
det \Big(\Big(\frac{\partial^2 \Delta P(x_0)}{\partial \tau_{l}\partial \tau_{j}}\Big)_{1\leq l,j\leq 2}\Big)\neq 0.
\end{equation*}

\begin{Thm}\label{nth1.2}
 Under the assumption \textup{($P$)},
   if $b_{0}\in \Gamma$ is  a non-degenerate critical point of $P(x)$ on $\Gamma$
   satisfying \eqref{1.6-1223},
  then there exists a large constant $a_{0}>0,$ such that \eqref{8-18-1}--\eqref{8-28-3}
  has a single-peak solution $u_a$ concentrating at $b_0$ as $a\in [a_{0},+\infty).$
\end{Thm}

The existence result in Theorem~\ref{nth1.2} is new   even without
the $L^{2}$-norm constraint
since our potential includes the degenerate case.
 Also if $P(x)$ does not achieve its global minimum on $\Gamma,$ any
 solution concentrating at a point on $\Gamma$ is not a ground state.
 So our existence result can not be obtained by the method in \cite{DGL-18JMP,GLW-18JMP,LZ-19ZAMP,LXZ-16-JMP}.

\smallskip
To state the local uniqueness result of single-peak solutions for
\eqref{8-18-1}--\eqref{8-28-3}, we give another assumption
$(\tilde{P})$ of $P(x)$:
if $ b_0$ is non-degenerate
and
\[
\Big(\frac{\partial^2 \Delta P(b_0)}{\partial \tau_{l} \partial\tau_{j}}\Big)_{1\leq l,j\leq 2}+\frac{\partial  \Delta P(b_0)}{\partial \nu}  diag \big(\kappa_{1}, \kappa_{2}\big)
\]
is non-singular, where $\kappa_{j}$ is the $j$-th principal curvature of $\Gamma$ at $b_0$ for $j=1,2.$

Our second main result is the following.

\begin{Thm}\label{th1.4}

Suppose
that the assumptions \textup{($P$)} and $(\tilde{P})$ hold.
Let $u_a^{(1)}(x)$ and $u_a^{(2)}(x)$ be two single-peak solutions of \eqref{8-18-1}--\eqref{8-28-3}
concentrating  at $b_0$ with $b_0\in \Gamma.$
Then there exists a large positive number $a_{0},$  such that $u_a^{(1)}(x)\equiv u_a^{(2)}(x)$ for all $a$ with
 $a_{0}\leq a <+\infty.$
\end{Thm}

\begin{Rem}
Our method can also be used to study the following Choquard equation with the dimensions $N=4,5$
\begin{equation}\label{ll2}
 -\epsilon^2\Delta u+P(x)u=\frac{a}{8\pi \epsilon^2}
 \big(\int_{\R^N}\frac{u^2(y)}{|x-y|^{N-2}}dy\big)u+\mu u,~x\in \R^N.
 \end{equation}
We would like to point out that in this case $a$ goes to
$
 \left\{
         \begin{array}{ll}
           a^{*}=\int_{\R^{4}}U^{2}(x)dx, & N=4, \vspace{0.12cm}\\
           0, & N=5.
         \end{array}
       \right.
$
\end{Rem}

The main difference in the discussion of the local uniqueness for \eqref{8-18-1} and \cite{LPW-20CV} is that the
Lagrange multiplier $\mu_{a}$ in \eqref{8-18-1} also depends on the solution $u_{a}.$ Hence its
corresponding linearized operator
has changed. For the details, see Lemma \ref{lem-bbb}.
Fortunately, such change does not bring much difficulties.

To our knowledge, except \cite {LPW-20CV}, there seems
to be no other local uniqueness results for peak (or bubbling) solutions
of a Schr\"odinger-Newton system.
Even for Schr\"odinger equations, there are very few results about
the local uniqueness of peak solutions.
The classical moving plane method is not still effective to study
the uniqueness of concentrating solutions.
 There are two main tools to investigate
the uniqueness of concentration solutions, i.e.,
 the topological degree method and Pohozaev identities.
Before the local uniqueness result of multi-bump solutions
in \cite{LPY-19},
other local uniqueness results for peak (or bubbling) solutions
of Schr\"odinger equations are available only  for the case when the solutions blow up
at  $x_0$,  an isolated critical point of the potential $P(x)$.
When $x_0$ is a non-degenerate critical point of $P(x)$, that is, $ (D^2 P)$ is non-singular at $x_0$, one
can prove the local uniqueness of the peak solution concentrating at $x_0$ either by counting the local degree of the
corresponding reduced finite dimensional  problem as in  \cite{Cao3,CNY,G}, or by using Pohozaev type identities as in  \cite{Cao1,Deng,Grossi,GPY,GLW}.

Compared with the topological degree method, it can deal with the degenerate potential to
 use the Pohazaev identities to prove the local uniqueness, one can refer to \cite{Cao1,Deng,GPY}.
We want to point out that in  \cite{Cao1,Deng,GPY}, though the critical point $x_0$ is degenerate,
 the rate of degeneracy along each direction is the same.
Moreover, in \cite{Grossi} Grossi gave an example which shows that
local uniqueness may not be true at a degenerate critical point $x_0$ of $P(x).$
When peak solutions concentrates at a degenerate critical point,
it is more subtle to study the local uniqueness of it. One can refer to \cite{LPY-19}.
Under our assumption ($P$), the potential $P(x)$ is non-degenerate along the normal direction $\nu$ of $\Gamma.$
But along each tangential direction of $\Gamma,$
$P(x)$ is degenerate. Such non-uniform degeneracy causes the estimates more complicated.
Moreover, there are two terms involving
volume integral in the corresponding local Pohozaev identity,
which brings us some difficulties.

Finally, we give a non-existence of multi-peak solutions for  \eqref{8-18-1}--\eqref{8-28-3}.
\begin{Thm}\label{th1.5}
Under the assumption \textup{($P$)}, assume that $b_{i}\in\Gamma_i(i=1,...,m)$
are non-degenerate critical points of $P(x)$ on $\Gamma_i$ satisfying
 \begin{equation*}\label{1.6}
(D_{\tau_{i,j}} \Delta  P)(b_i)=0,~\mbox{with}~
 i=1,\cdots,m~\mbox{and}~j=1,2,
\end{equation*}
where $\tau_{i,j}$ is the $j$-th  principal tangential unit vector of $\Gamma$ at $b_i.$
  Then there exists a large constant $a_{0}>0,$ such that
 problem  \eqref{8-18-1}--\eqref{8-28-3}
has no $m$-peaks solutions ($m\geq 2$) of this form
\begin{equation}\label{20-23-5}
u_a (x)= \frac{-\mu_a}{\sqrt{a}} \sum_{i=1}^m \Bigl( U\bigl( \sqrt {-\mu_a}(x-x_{a,i})\bigr)+\varpi_a(x)\Bigr),
\end{equation}
with $\displaystyle\int_{\R^3}\big[-\frac{1}{\mu_a}|\nabla \varpi_a|^2+\varpi_a^2\big]=o\big(\frac{1}{(-\mu_a)^{\frac{3}{2}}}\big)$ and  some points $x_{a,i}\in \mathbb R^3(i=1,\cdots,m)$ satisfying $x_{a,i}\rightarrow b_{i}, i=1,...,m$ as $a\to +\infty$ and $b_i\neq b_j$ for $i\neq j.$
\end{Thm}

\begin{Rem}
The nonexistence result of multi-peak solutions for  \eqref{8-18-1}--\eqref{8-28-3}
is very different from the Schr\"odinger problem studied in \cite{LPY-19},
 which is mainly caused by the nonlocal term.
\end{Rem}

To prove Theorem \ref{th1.5}, we mainly use some local Pohozaev identities and some contradiction argument.
We have to obtain some accurate estimates involved by the nonlocal term as \cite{LPW-20CV}.

This paper is organized as follows.
 We  prove Theorem~\ref{th1-24-5} in section~\ref{s2}. In section~\ref{s3}, we estimate the Lagrange multiplier $\mu_a$
in terms of $a.$ In  section~\ref{s4}
we prove the results for the location of the peaks and for the existence of single-peak solutions. We study the local uniqueness of
single-peak solutions in section~\ref{s5}. In section \ref{s6}, we prove the nonexistence result of multi-peak solutions.
We put some known results and some basic and technical estimates in Appendices \ref{sa} to \ref{sd}.

  For simplicity, we use
$|u|_{q}(2\leq q\leq6)$ to denote $\big(\int_{\R^{3}}|u(x)|^{q}dx\big)^{\frac{1}{q}}$ and $\|u\|$  the usual $H^{1}(\R^{3})$ norm.
In this paper, we always assume that $b_0\in \Gamma_.$

\section{The proof of Theorem~\ref{th1-24-5}}\label{s2}

First, we study the following problem:
\begin{equation}\label{1-26-10}
-\Delta u =   P_1(x)u, \; u>0, \quad \text{in}\; \mathbb R^3,
\end{equation}
where the function $P_1(x)$ satisfies $P_1> 1$  in $B_R(0)\setminus B_{t}(0)$ for some fixed $t>0$ and large $R>0$.

In order to prove Theorem~\ref{th1-24-5}, first we give the following result.
\begin{Prop}\label{th1-26-10}(Proposition 2.1, \cite{LPY-19})
Problem~\eqref{1-26-10} has no solution.
\end{Prop}

With Proposition \ref{th1-26-10} at hand, now we are in a position to prove Theorem \ref{th1-24-5}.
\begin{proof}[\textbf{Proof of Theorem~\ref{th1-24-5}}]

First, we prove that $\mu_a\to -\infty$. We argue by an indirect method.
Suppose that  $|\mu_a|\leq M$.  Since  $\int_{\mathbb R^3} u_a^2=1,$
by the Moser iteration, we can prove that $u_a$ is uniformly bounded. That is, $u_a$ does not blow up.

 Suppose that  $\mu_a\rightarrow +\infty.$  Set $P_1(x)= \mu_a - P(x) + \frac{a}{8\pi}\int_{\R^{3}} \frac{u^{2}(y)}{|x-y|}dy.$
   Noting that $u_a$ concentrates at some points, we have for $x\in \mathbb R^3\setminus B_t(0)$ and $t>0$
   \begin{equation}\label{eq-7-9-1}
 \begin{split}
   & \int_{\R^{3}} \frac{u_{a}^{2}(y)}{|x-y|}dy
    \\
    &=\int_{|x-y|\leq \frac{t}{2}} \frac{u_{a}^{2}(y)}{|x-y|}dy
    +\int_{|x-y|\geq \frac{t}{2}} \frac{u_{a}^{2}(y)}{|x-y|}dy\\
    &=o(1)\int_{|x-y|\leq \frac{t}{2}} \frac{1}{|x-y|}dy+\Big(\int_{|x-y|\geq \frac{t}{2}} { |u_{a}(y)|^{\frac{20}{7}}}dy\Big)^{\frac{7}{10}}
    \Big( \frac{1}{|x-y|^{\frac{10}{3}}}dy\Big)^{\frac{3}{10}}\\
    &=o(1)t^{2}+O\Big(\frac{1}{t^{\frac{1}{10}}}\Big),
    \end{split}
 \end{equation}
 where $o(1)$ goes to zero as $a\rightarrow \infty.$
Hence from \eqref{eq-7-9-1} we  may suppose that $\frac{a}{8\pi}\int_{\R^{3}} \frac{u_{a}^{2}(y)}{|x-y|}dy\ge -1$ in
$\mathbb R^3\setminus B_t(0)$ for some $t>0.$  Hence for any fixed $R>0,$ we usually have

\[
P_1(x)= \mu_a - P(x) + \frac{a}{8\pi}\int_{\R^{3}} \frac{u_{a}^{2}(y)}{|x-y|}dy>1,\quad x\in B_R(0)\setminus B_t(0).
\]
From  Proposition~\ref{th1-26-10}, we obtain a contradiction.
So we have proved that $\mu_a\to -\infty.$  Let  $\lambda_a=-\mu_a.$
Let $x_a$ be the maximum point of $u_a$. By equation \eqref{8-18-1}, we find

\[
\frac{a}{8\pi}\int_{\R^{3}} \frac{u_{a}^{2}(y)}{|x-y|}dy\,u_{a}(x) \ge \big(\lambda_a + P(x_a)\big) u_a(x_a)>0,
 \]
 which implies that $a>0.$

Denote $\bar u_a (x)= \frac1{\lambda_a} u_a\bigl( \frac x{\sqrt{\lambda_a}}\bigr).$
Then
\begin{equation}\label{30-27-5}
-\Delta \bar u_a + \Bigl( 1  + \frac1{ {\lambda_a}} P\bigl( \frac x{\sqrt{\lambda_a}}\bigr)\Bigr)\bar u_a  =
\frac{a }{8\pi}\int_{\R^{3}}\frac{\bar u_a^2(y)}{|x-y|}dy\,\bar{u}_a(x),\quad \text{in}\; \mathbb R^3,
\end{equation}
and
\begin{equation}\label{31-27-5}
\int_{\mathbb R^3} \bar u_a^2=\frac{1}{\lambda^{\frac{1}{2}}_{a}}.
\end{equation}

 By \eqref{30-27-5} and \eqref{31-27-5}, applying Moser iteration,
 we can show that $|u_a|\le C$ for some constant independent of $a.$  Let $\bar x_a$ be a maximum point of $\bar
 u_a$.  Then

\[
\frac{a }{8\pi}\int_{\R^{3}}\frac{\bar u_a^2(y)}{|\bar x_a-y|}dy\,\bar{u}_a(\bar x_a)\ge \Bigl( 1  + \frac1{ {\lambda_a}} P\bigl( \frac {\bar x_a}{\sqrt{\lambda_a}}\bigr)\Bigr)\bar u_a(\bar x_a),
\]
 which gives $
 a\ge   8\pi\Big(\int_{\R^{3}}\frac{\bar u_a^2(y)}{|\bar x_a-y|}dy\Big)^{-1} \ge c_0>0.$
Applying the standard blow-up argument, by \eqref{31-27-5} we can check that there holds
\begin{equation}\label{32-27-5}
 \bar u_a = U_{a} (x-\bar x_{a})+\overline{\varpi}_a(x),
 \end{equation}
 for some $\bar x_{a} \in \mathbb R^3$ with
 $$
 \int_{\R^3}\big[|\nabla\overline{\varpi}_a|^2
 +(\overline{\varpi}_{a})^2\big]=o(1),
 $$
and $U_{a}$ is the unique positive solution of

 \[
 -\Delta u + u = \frac{a}{8\pi} \int_{\R^{3}}\frac{u^2(x)}{|x-y|}dy\,u,\;\; u\in H^1(\mathbb R^3),\;\; u(0)=\max_{x\in \mathbb R^3} u(x).
 \]
 Observing that $U_{a}=\frac1{\sqrt{a}} U,$  it follows from \eqref{31-27-5}  and \eqref{32-27-5} that
 $a\rightarrow a^{*}\lambda^{\frac{1}{2}}_{a}\rightarrow +\infty.$
\end{proof}

\section{Some estimates for general potentials}\label{s3}

In this section,  we shall estimate $\mu_a$ respect to $a.$

Let $\epsilon =\frac{1}{ \sqrt{-\mu_a}}$ and $u(x)\mapsto \frac{-\mu_a}{\sqrt{a}} u(x).$
Then \eqref{8-18-1} can be rewritten as
\begin{equation}\label{30-7-11}
-\epsilon^2\Delta u+ \bigl( 1+ \epsilon^2 P(x)\bigr)u= \frac{1}{8\pi \epsilon^{2}}\int_{\R^{3}}\frac{u^{2}(y)}{|x-y|}dy\,u(x),~
u\in H^1(\R^3).
\end{equation}
For any $a\in \R^+$, we define
 $\|u\|_a:=\displaystyle\int_{\mathbb R^3} \bigl( \epsilon^2 |\nabla u|^2 + u^2\bigr)^{\frac{1}{2}}$.

 By \eqref{20-23-5-1223}, we know that a single-peak solution of \eqref{30-7-11} has the following form
\begin{equation*}
\tilde{u}_a(x)= U_{\epsilon,x_{a}}(x) +\varphi_{a}(x),~\mbox{with}~~\|\varphi_{a}\|_a=o(\epsilon^{\frac 32}),
\end{equation*}
where $U_{\epsilon,x_{a}}(x):=(1+\epsilon^{2}P_{0})U\Big(\frac{\sqrt{1+\epsilon^2P_0}(x-x_{a})}\epsilon \Big)$.
Then, there holds
 \begin{equation}\label{06-07-3}
 \begin{split}
   & -\epsilon^2\Delta \varphi_{a}+\Big((1+\epsilon^2P(x))\varphi_{a}-
    \frac{1}{4\pi\epsilon^{2}}\int_{\R^{3}}\frac{U_{\epsilon,x_{a}}(y)\varphi_{a}(y)}{|x-y|}dyU_{\epsilon,x_{a}}(x)
    \\
    &\quad\quad-\frac{1}{8\pi\epsilon^{2}}\int_{\R^{3}}\frac{(U_{\epsilon,x_{a}}(y))^{2}}{|x-y|}dy\varphi_{a}(x)\Big)
    =\mathcal {R}_{a,\epsilon}\big(\varphi_{a}\big)+\mathcal {L}_{a}(x),
    \end{split}
 \end{equation}
where
\begin{equation}\label{06-07-1}
\begin{split}
      \mathcal {R}_{a,\epsilon}\big(\varphi_{a}\big)&= \Big(\frac{1}{8\pi \epsilon^{2}}
      \int_{\R^{3}}\frac{(U_{\epsilon,x_{a}}+\varphi_{a})^{2}(y)}{|x-y|}dy\,(U_{\epsilon,x_{a}}+\varphi_{a})(x)
    -\frac{1}{8\pi \epsilon^{2}}
      \int_{\R^{3}}\frac{(U_{\epsilon,x_{a}}(y))^{2}}{|x-y|}dy\,U_{\epsilon,x_{a}}(x)
      \\
      &\quad-\frac{1}{4\pi \epsilon^{2}}
      \int_{\R^{3}}\frac{(U_{\epsilon,x_{a}}\varphi_{a})(y)}{|x-y|}dy\,U_{\epsilon,x_{a}}(x)
      \Big)
      -\frac{1}{8\pi \epsilon^{2}}
      \int_{\R^{3}}\frac{(U_{\epsilon,x_{a}})^{2}(y)}{|x-y|}dy\,\varphi_{a}(x)\\
      &=\frac{1}{4\pi \epsilon^{2}}
      \int_{\R^{3}}\frac{(U_{\epsilon,x_{a}}\varphi_{a})(y)}{|x-y|}dy\,\varphi_{a}(x)
      +\frac{1}{8\pi \epsilon^{2}}
      \int_{\R^{3}}\frac{(\varphi_{a})^{2}(y)}{|x-y|}dy\,U_{\epsilon,x_{a}}(x)
      \\
      &\quad+\frac{1}{8\pi \epsilon^{2}}
      \int_{\R^{3}}\frac{(\varphi_{a})^{2}(y)}{|x-y|}dy\,\varphi_{a}(x),
      \end{split}
\end{equation}
and
\begin{equation*}\label{eq-l}
\begin{split}
\mathcal {L}_{a}(x)=-\epsilon^2\big(P(x)-P_0\big)U_{\epsilon,x_{a}}(x).
 \end{split}
\end{equation*}
 We can move $x_{a}$ a bit(still denoted by $x_{a}$), so that the error term $\varphi_a\in \displaystyle E_{a,x_{a}},$ where
\begin{equation*}\label{06-08-1}
E_{a,x_{a}} :=\left \{ u(x)\in H^1(\R^3):
\Big\langle u,\frac{\partial U_{\epsilon,x_{a}}(x)}{\partial{x_j}}
\Big\rangle_{a}=0, ~j=1,2,3
 \right\}.
 \end{equation*}

 Let  $L_{a}$ be the bounded linear operator from  $H^1(\mathbb R^3)$ to itself, defined by
\begin{equation*}\label{06-07-2}
\bigl\langle L_{a} u, v\bigr\rangle_a=\int_{\mathbb R^3}  \big(  \epsilon^2\nabla  u\nabla v
 +\big(1+\epsilon^2P(x)\big)u v\big)-\frac{3}{8\pi \epsilon^{2}}
     \int_{\R^{3}} \int_{\R^{3}}\frac{u^{2}(x)(uv)(y)}{|x-y|}dxdy.
\end{equation*}
Then, it is standard to prove the following lemma.
\begin{Lem}\label{lem-A.1}
There exist  constants $\varrho>0$ and large $a_{0}>0$ such that for all $a$ with $a_{0}\leq a <+\infty,$  it holds
 \begin{equation}\label{B.2}
   \|L_{a}u\|_{a}\geq \varrho \|u\|_{a},~\mbox{for all}~u\in E_{a,x_{a}}.
 \end{equation}
\end{Lem}

\begin{Lem}
There exists a constant $C>0$ independent of $a$ such that
\begin{equation}\label{aaab10-31-7}
    \|\mathcal {R}_{a,\epsilon}\big(\varphi_{a}\big)\|_{a}=O\big(\epsilon^{-3}\|\varphi_{a}\|^{3}_{a}
    +\epsilon^{-\frac{3}{2}}\|\varphi_{a}\|^{2}_{a}).
\end{equation}
\end{Lem}

\begin{proof}
From \eqref{06-07-1}, for any $\upsilon\in H^{1}(\R^{3})$ we have
\begin{equation}\label{07-09-11}
\begin{split}
      &(\mathcal {R}_{a,\epsilon}\big(\varphi_{a}\big),\upsilon)\\
      &=\frac{1}{4\pi \epsilon^{2}}
    \underbrace{ \int_{\R^{3}} \int_{\R^{3}}\frac{\varphi_{a}(x)\upsilon(x)(U_{\epsilon,x_{a}}\varphi_{a})(y)}{|x-y|}dx\,dy}_{:=A_{1}}
      +\frac{1}{8\pi \epsilon^{2}}
     \underbrace{ \int_{\R^{3}}\int_{\R^{3}}\frac{U_{\epsilon,x_{a}}(x)\upsilon(x)(\varphi_{a})^{2}(y)}{|x-y|}dx\,dy}_{:=A_{2}}
      \\
      &\quad+\frac{1}{8\pi \epsilon^{2}}
      \underbrace{\int_{\R^{3}}\int_{\R^{3}}\frac{\varphi_{a}(x)\upsilon(x)(\varphi_{a})^{2}(y)}{|x-y|}dx\,dy}_{:=A_{3}}.
      \end{split}
\end{equation}
By Lemma \ref{lem-7-9-1}, we have
\begin{equation*}\label{07-09-12}
A_{1}\leq C\epsilon^{-1}\|\varphi_{a}\|^{2}_{a}\|\upsilon\|_{a}\|U_{\epsilon,x_{a}}\|_{a}\leq C\epsilon^{\frac{1}{2}}\|\varphi_{a}\|^{2}_{a}\|\upsilon\|_{a},
\,\,A_{2}\leq  C\epsilon^{\frac{1}{2}}\|\varphi_{a}\|^{2}_{a}\|\upsilon\|_{a},
\end{equation*}
and
\begin{equation}\label{07-09-13}
A_{3}\leq C\epsilon^{-1}\|\varphi_{a}\|^{3}_{a}\|\upsilon\|_{a}.
\end{equation}
It follows from \eqref{07-09-11} and \eqref{07-09-13} that
\eqref{aaab10-31-7} holds.
\end{proof}

\begin{Lem}\label{lem-7-9-3}
There holds
 \begin{equation}\label{10-31-7}
   \|\mathcal {L}_{a}\|_{a}=O\Big(\big|\big(P(x_{a})-P_0\big)\big|\epsilon^{{\frac{7}{2}} } +\big| \nabla P(x_{a})
   \big|\epsilon^{\frac{9}{2}}+ \epsilon^{\frac{11}{2}} \Bigr).
 \end{equation}
\end{Lem}

 \begin{proof}
 For any $\upsilon\in H^{1}(\R^{3}), $ we have
\begin{equation*}\label{7-9-2-1}
\begin{split}
(\mathcal {L}_{a}(x),\upsilon)&=-\epsilon^2\int_{\R^{3}}\big(P(x)-P_0\big)U_{\epsilon,x_{a}}(x)\upsilon(x) dx\\
&=-\epsilon^2\underbrace{\int_{\R^{3}}\big(P(x)-P(x_{a})\big)U_{\epsilon,x_{a}}(x)\upsilon(x) dx}_{:=B_{1}}
-\epsilon^2\underbrace{\int_{\R^{3}}\big(P(x_{a})-P_{0}\big)U_{\epsilon,x_{a}}(x)\upsilon(x) dx}_{:=B_{2}}.
 \end{split}
\end{equation*}
By H\"older inequality and the exponential of $U(x)$ at infinity, we can estimate
\begin{equation*}\label{7-9-2-2}
\begin{split}
B_{1}
&=\int_{B_{\delta}(x_{a})}\big(P(x)-P(x_{a})\big)U_{\epsilon,x_{a}}(x)\upsilon(x) dx
+\int_{B^{C}_{\delta}(x_{a})}\big(P(x)-P(x_{a})\big)U_{\epsilon,x_{a}}(x)\upsilon(x) dx\\
&=\int_{B_{\delta}(x_{a})}(\nabla P(x_{a})\cdot (x-x_{a})+O(|x-x_{a}|^{2})U_{\epsilon,x_{a}}(x)\upsilon(x) dx
+C\Big(\int_{B^{C}_{\delta}(x_{a})}U^{2}_{\epsilon,x_{a}}(x) dx\Big)^{\frac{1}{2}}|\upsilon|_{2}\\
&\leq |\nabla P(x_{a})|\Big(\int_{B_{\delta}(x_{a})}|x-x_{a}|^{2}U^{2}_{\epsilon,x_{a}}(x)dx\Big)^{\frac{1}{2}}|\upsilon|_{2}
+C\Big(\int_{B_{\delta}(x_{a})}|x-x_{a}|^{4}U^{2}_{\epsilon,x_{a}}(x)dx\Big)^{\frac{1}{2}}|\upsilon|_{2}
\\
&\quad+O(e^{-\frac{\theta}{\epsilon}})\|\upsilon\|_{a}\\
&=O(|\nabla P(x_{a})|\epsilon^{\frac{5}{2}}\|\upsilon\|_{a})+O(\epsilon^{\frac{7}{2}}\|\upsilon\|_{a}),
 \end{split}
\end{equation*}
and
\begin{equation*}\label{7-9-2-3}
\begin{split}
B_{2}
&=\int_{\R^{3}}\big(P(x_{a})-P_{0}\big)U_{\epsilon,x_{a}}(x)\upsilon(x) dx\\
&\leq |P(x_{a})-P_{0}|\Big(\int_{\R^{3}}U^{2}_{\epsilon,x_{a}}(x)dx\Big)^{\frac{1}{2}}|\upsilon|_{2}
=O(|P(x_{a})-P_{0}|\epsilon^{\frac{3}{2}})\|\upsilon\|_{a}.
 \end{split}
\end{equation*}
Combining all the estimates above, we know that \eqref{10-31-7} holds.
 \end{proof}

 With Lemmas \ref{lem-A.1} to \ref{lem-7-9-3} at hand, we can prove
\begin{Lem}\label{lem-5-05-1}
A single-peak solution $\tilde{u}_a$ for \eqref{30-7-11} concentrating at $b_0$ has the following form
\begin{equation}\label{aaa8-20-1}
\tilde{u}_a(x)= U_{\epsilon,x_{a}}(x)+\varphi_{a}(x),
\end{equation}
with  $\varphi_a \in \displaystyle E_{a,x_{a}}$  and
\begin{equation}\label{lt1}
\|\varphi_{a}\|_a=O\Big(
\big|P(x_{a})-P_0\big| \epsilon^{{\frac{7}{2}} } +\big| \nabla P(x_{a})
   \big|\epsilon^{\frac{9}{2}}+ \epsilon^{\frac{11}{2}} \Bigr).
\end{equation}
\end{Lem}

\begin{proof}
  Then from \eqref{06-07-3}, \eqref{B.2}, \eqref{aaab10-31-7}, \eqref{10-31-7} and
  by applying the contraction mapping theorem, we can get  \eqref{aaa8-20-1} and \eqref{lt1}  using  the standard argument.

\end{proof}

Let  $\tilde \varphi_{a}(x) =\varphi_{a}(\epsilon x + x_{a}).$ Then, $\tilde \varphi_{a}$ satisfies
$\|\tilde \varphi_{a}\|_a=O\big(\epsilon^2\big).$ Using the Moser iteration,
we can prove
$\|\tilde \varphi_{a}\|_{L^\infty(\mathbb R^3)}=o(1).$
 From this  and the comparison theorem, similar to Proposition 2.2 in \cite{LPW-20CV},
  we can prove the following estimates for $\tilde u_a(x)$ away from the concentrated point $b_0.$

\begin{Prop}
Suppose that $\tilde u_a(x)$ is a single-peak solution of \eqref{30-7-11} concentrating at $b_0.$ Then
for any fixed $R\gg 1,$ there exist some $\theta>0$ and $C>0,$ such that
\begin{equation}\label{2--5}
|\tilde u_a(x)|+|\nabla \tilde u_a(x)|\leq C e^{-\theta |x-x_{a}|/\epsilon},~\mbox{for}~
x\in \R^3\backslash B_{R \epsilon}(x_{a}).
\end{equation}
\end{Prop}

\begin{Lem}
There holds
\begin{equation}\label{06-13-1}
\displaystyle\int_{\R^3} U^2=\frac{3}{32\pi}\int_{\R^3}\int_{\R^3}\frac{U^2(x)U^2(y)}{|x-y|}dx\,dy.
\end{equation}
\end{Lem}

\begin{proof}
It follows directly from the following two identities:
\begin{equation*}\label{a5-21-1}
 \int_{\R^3}(|\nabla U|^2+ U^2)
=\frac{1}{8\pi}\int_{\R^3}\int_{\R^3} \frac{U^2(x)U^2(y)}{|x-y|}dx\,dy
\end{equation*}
and
\begin{equation*}\label{eq2}
\int_{\R^3}|(\nabla U|^2+ 3U^2)
=\frac{5}{16\pi}\int_{\R^3}\int_{\R^3} \frac{U^2(x)U^2(y)}{|x-y|}dx\,dy,
\end{equation*}
where the last equality can be deduced by multiplying $\langle x, \nabla U\rangle$ on both sides of \eqref{big-u} and integrating on $\R^{3}.$
\end{proof}

 \begin{Prop}
Letting $a\rightarrow +\infty,$
there holds
\begin{equation}\label{ab8-29-1}
\frac{a }{\sqrt{-\mu_a}}= a_*
-\frac{P_0}{2\mu_a}a_*+O\Big(\big|P(x_{a})-P_{0}\big|\frac{1}{
 (-\mu_a)}
  +\big|\nabla P(x_{a})
   \big|\frac{1}{(\sqrt{-\mu_a})^3 }+ \frac{1}{\mu_a^2 } \Bigr).
\end{equation}
\end{Prop}
\begin{proof}
From \eqref{8-18-1} and \eqref{8-28-3}, we have
\begin{equation*}
\begin{split}
1=&\int_{\R^3}\big({u}_a(x)\big)^2=
\int_{\R^3}\Big(
 \bar U_{a,x_{a}}(x)+\frac{(-\mu_{a})}{\sqrt{a}}\varphi_{a}(x)\Big) ^2,
\end{split}
\end{equation*}
where $\bar{U}_{a,x_{a}}(x)=\frac{-\mu_a+P_0}{\sqrt{a}}U(\sqrt{-\mu_a+P_0}(x-x_{a})).$
By direct computation, we can obtain
\begin{equation*}
\begin{split}
a=a_*\sqrt{-\mu_a+P_0} +O\Big(\big|P(x_{a})-P_0\big|
 \frac{1}{\sqrt{-\mu_a} }+\big| \nabla P(x_{a})
   \big|\frac{1}{(-\mu_a)}+\frac{1}{(\sqrt{-\mu_{a}})^3}\Bigr),
\end{split}
\end{equation*}
which implies that
\eqref{ab8-29-1} is true.
\end{proof}

\section{ Locating the peak and the existence of single-peak solutions}\label{s4}

 First, we locate the peak for a single-peak solution. Let $\tilde{u}_a$ be a single-peak solution of \eqref{30-7-11}.
Then for any small fixed $\rho>0$, from \eqref{7-30-1} and \eqref{eqd-7-10-1} we find
\begin{equation}\label{30-31-7}
\begin{split}
\epsilon^2&\int_{B_{\rho}(x_{a})} \frac{\partial P(x)}{\partial x_j}\big(\tilde{u}_a\big)^2  dx
\\=&\underbrace{-2\epsilon^2\int_{\partial B_{\rho}(x_{a})}\frac{\partial \tilde{u}_a}{\partial \bar \nu}\frac{\partial \tilde{u}_a}{\partial x_j} d\sigma}_{:=C_{1}}
+\underbrace{\epsilon^2\int_{\partial B_{\rho}(x_{a})}|\nabla \tilde{u}_a|^2\bar\nu_j(x) d\sigma}_{:=C_{2}}\\&
 +\underbrace{\int_{\partial B_{\rho}(x_{a})}
\big(1+\epsilon^2P(x)\big)\big(\tilde{u}_a\big)^2\bar\nu_j(x)d\sigma}_{:=C_{3}}
-\underbrace{\frac{1}{8\pi\epsilon^{2}}\int_{\partial B_{\rho}(x_{a})}\int_{\R^{3}}
\frac{(\tilde{u}_a)^2}{|x-y|}dy(\tilde{u}_a)^2\bar\nu_j(x)d\sigma}_{:=C_{4}}
\\\quad&
-\underbrace{\frac{1}{8\pi\epsilon^{2}}\int_{ B_{\rho}(x_{a})}\int_{\R^{3}}
\frac{(\tilde{u}_a(y))^2(\tilde{u}_a(x))^2(x_{j}-y_{j})}{|x-y|^{3}}dydx}_{:=C_{5}}
\\
=&
O(e^{-\frac{\theta}{\epsilon}}), ~
\mbox{with some}~\theta>0,
\end{split}
\end{equation}
where $j=1,2,3$ and $\bar \nu(x)=\big(\bar\nu_{1}(x),\bar\nu_2(x),\bar\nu_3(x)\big)$ is the outward unit normal of $\partial B_{\rho}(x_{a}).$
And then \eqref{30-31-7} implies the first necessary condition for the concentrated point $b_0:$
$$
\nabla P(b_0)=0.
$$
Now we are in a position to prove Theorem \ref{nth1.1}.

\begin{proof}[\textbf{Proof of Theorem \ref{nth1.1}}]

Since $x_{a}\to b_0\in \Gamma,$ we find that there is a $t_a\in [P_{0}, P_{0}+\sigma]$ if $\Gamma$
is a local minimum set of $P(x)$, or $t_a\in [P_{0}-\sigma, P_{0}]$ if $\Gamma$
is a local maximum set of $P(x)$, such that $x_{a}\in \Gamma_{t_a}.$
Let $\tau_{a}$ be the unit tangential vector of $\Gamma_{t_a}$ at $x_{a}.$  Then
$$
G(x_{a}) =0,~\mbox{where}~G(x)=  \bigl\langle \nabla P(x), \tau_{a}\bigr\rangle.
$$
We have the following expansion:
\begin{equation*}
\begin{split}
 G(x)=& \langle\nabla G( x_{a}), x-x_{a}\rangle+
\frac{1}{2}\big\langle \langle \nabla^2 G( x_{a}), x- x_{a}\rangle,x- x_{a}\big\rangle
+o\big(|x- x_{a}|^2\big),~\mbox{for}~x\in B_{\rho}(x_{a}).
\end{split}
\end{equation*}
Then it follows from \eqref{lt1}, \eqref{30-31-7} and  the above expansion that
\begin{equation}\label{luo-6}
\begin{split}
\int_{\mathbb R^3} G(x)U^{2}_{\epsilon,x_{a}}(x)&=\int_{B_{\rho}(x_{a})} G(x)U^{2}_{\epsilon,x_{a}}(x)
+O\big(e^{-\frac{\theta}{\epsilon}}\big)
\\&=
-2\int_{B_{\rho}(x_{a})} G(x)U_{\epsilon,x_{a}}(x) \varphi_{a}
- \int_{B_{\rho}(x_{a})}G(x) \varphi^2_{a}+O\big(e^{-\frac{\theta}{\epsilon}}\big)\\&=
O\Big( \big[\epsilon^{\frac{5}{2}} |\nabla G(x_{ a})| +\epsilon^{\frac{7}{2}}\big]\|\varphi_{a}\|_a+\epsilon|\nabla G(x_{ a})|\cdot
\|\varphi_{a}\|^2_a\Big)+O\big(e^{-\frac{\theta}{\epsilon}}\big)\\
&=
O\big( \epsilon^{6}\big).
\end{split}
\end{equation}
On the other hand, noting that $ G(x_{a})=0,$ it is easy to check
\begin{equation}\label{06-09-1}
\int_{\mathbb R^3} G(x)U^{2}_{\epsilon,x_{a}}(x)=
\frac{1}{6} \epsilon^{5}\Delta G(x_{a}) \displaystyle\int_{\R^3}|x|^2U^2
+O\bigl (\epsilon^{7}\bigr).
\end{equation}
Then it follows from \eqref{luo-6} and \eqref{06-09-1} that $(\Delta G) (x_{a})=O(\epsilon).$
Thus by the assumption $(P),$  we get \eqref{1.6-1223}.
\end{proof}

Now, we study the  existence of single-peak solutions for  \eqref{1-23-5} with $\lambda>0$ a large parameter.
Letting $\eta=\frac{1}{\sqrt{\lambda}}$ and $w(x)\mapsto \lambda w(x),  $ then \eqref{1-23-5}
 can be changed to the following problem:
\begin{equation}\label{20-7-11}
-\eta^2\Delta w+ \bigl( 1+ \eta^2 P(x)\bigr)w=\frac{1}{8\pi\eta^{2}}\int_{\R^{3}}\frac{w^2}{|x-y|}dy\,w,~
w\in H^1(\R^3).
\end{equation}

In the sequel, we denote $\langle u, v\rangle_\eta = \int_{\mathbb R^3} \bigl( \eta^2 \nabla u\nabla v   +  uv\bigr)$
and $\|u\|_\eta=\langle u, u\rangle_\eta^{\frac12}.$
Now for $\eta>0$ small, we construct a single-peak solution $u_\eta$ of \eqref{20-7-11}  concentrating at
$b_0.$
Here  we can prove the following result in a standard way.

\begin{Prop}
 There is an $\eta_0>0$, such that for any $\eta\in (0, \eta_0]$, and $z_0$ close to $b_0,$
   there exists $v_{\eta, z_{0}}\in  F_{\eta,z_{0}}$ such that
\begin{equation*}
\begin{split}
\int_{\mathbb R^3}  &\bigl( \eta^2 \nabla w_{\eta}  \nabla \psi  +  \bigl( 1+ \eta^2 P(x)\bigr)w_{\eta} \psi
 =\frac{1}{8\pi\eta^{2}}\int_{\mathbb R^3}\int_{\R^{3}}\frac{w_{\eta}^2(y)}{|x-y|}w_{\eta}(x) \psi(x) dy\,dx,~~
\mbox{for all}~\psi \in  F_{\eta,z_{0}},
\end{split}
\end{equation*}
where
\begin{equation}\label{solu-2}
w_{\eta}(x)=\displaystyle U_{\eta,z_0}(x)+  \varphi_{\eta, z_{0}}(x)
\end{equation}
and
\begin{equation*}
 F_{\eta,z_{0}}=\left \{ u(x)\in H^1(\R^3):
\Big\langle u,\frac{\partial U_{\eta,z_{0}}(x)}{\partial{x_j}}
\Big\rangle_{\eta}=0, ~j=1,2,3
 \right\}.
\end{equation*}
 Moreover, it holds
\begin{equation}\label{eq-fhi}
\|\varphi_{\eta, z_{0}}\|_\eta=O\big(
\big| P(z_0)-P_0\big|\eta^{\frac{7}{2}}+\big|\nabla P(z_0)\big|\eta^{\frac{9}{2}}+\eta^{\frac{11}{2}}\big).
\end{equation}

\end{Prop}

 To obtain a true solution for \eqref{20-7-11}, we have to choose $z_{0}$ such that
\begin{equation*}
\begin{split}
\int_{B_\rho(x_{a})}  \Bigl(-\eta^2 \Delta w_{\eta} \frac{\partial w_\eta}{\partial x_j}  + \bigl( 1+ \eta^2 P(x)\bigr) w_{\eta} \frac{\partial w_\eta}{\partial x_j}-\frac{1}{8\pi\eta^{2}}\int_{\R^{3}}\frac{w_{\eta}^2(y)}{|x-y|}dy\,w_{\eta} \frac{\partial w_\eta}{\partial x_j}
\Bigr) =0,~~j=1,2,3.
 \end{split}
\end{equation*}
Similar to \eqref{30-31-7}, it is easy to check that the above identities are equivalent to
\begin{equation}\label{10-7-11}
\int_{B_\rho(x_{a})}  \frac{\partial P(x)}{\partial x_j} w^2_\eta=O\big(e^{-\frac{\theta}{\eta}}\big),\quad \forall ~~j=1,2,3.
\end{equation}

For $z_0$ close to $b_0,$  and $z_0\in \Gamma_{t}$ for some $t$ close to $P_0,$
now we use $\nu$ to denote the unit normal vector of $\Gamma_{t}$ at $z_0,$ while we use $\tau_{j}$ ($j=1, 2$)
 to denote the principal directions of $\Gamma_{t}$ at $x_{a}.$ Then, at $z_0,$ it holds
\[D_{\tau_{j}} P(z_0)=0,~\mbox{for}~
j=1, 2,~\mbox{and}~ |\nabla  P(z_0)|= |D_{\nu}  P(z_0)|.
 \]

We first prove the following result.
\begin{Lem}\label{lem-001-1-8}
Under the assumption \textup{($P$)}, $\displaystyle\int_{B_\rho(x_{a})} D_{\nu}  P(x) u^2_\eta
=O\big(e^{-\frac{\theta}{\eta}}\big) $ is equivalent to
\begin{equation}\label{00-1-1-8}
D_{\nu}  P(z_0)=O\bigl( \eta^2\bigr).
\end{equation}

\end{Lem}

\begin{proof}
First, from \eqref{solu-2} we have
\begin{equation}\label{00-aluo-6}
\begin{split}
\int_{\mathbb R^3}& D_{\nu}  P(x)U^{2}_{\eta,z_{0}}(x)\\=&
-2 \int_{\mathbb R^3} D_{\nu}  P(x)U_{\eta,z_{0}}(x)\varphi_{\eta,z_{0}}-\int_{\mathbb R^3}D_{\nu}  P(x)
  \varphi^2_{\eta,z_{0}}+O\big(e^{-\frac{\theta}{\eta}}\big) \\=&
O\big(|  D_{\nu}  P(z_0)| \eta^{\frac{3}{2}} \cdot\| \varphi_{\eta, z_{0}}\|_\eta+\eta^{\frac{5}{2}}\| \varphi_{\eta, z_{0}}\|_\eta+
\| \varphi_{\eta, z_{0}}\|^2_\eta\big)=
O\big(\eta^{5}\big).
\end{split}
\end{equation}
On the other hand,
we have
\begin{equation}\label{00-abcluo-3}
\begin{aligned}
\int_{\mathbb R^3}& D_{\nu}  P(x)U^{2}_{\eta,z_{0}}(x)=a_*\eta^{3}  D_{\nu}  P( z_0)
 +O\big(\eta^{5}\big).
\end{aligned}
\end{equation}
Then it follows form \eqref{00-aluo-6} and \eqref{00-abcluo-3} that \eqref{00-1-1-8} holds.

\end{proof}

\begin{Lem}\label{lem-002-1-8}
Under the assumption \textup{($P$)},  $\displaystyle\int_{B_\rho(x_{a})} D_{\tau}  P(x) u^2_\eta
=O\big(e^{-\frac{\theta}{\eta}}\big) $   is equivalent to
\begin{equation}\label{01-1-1-8}
(D_{\tau}  \Delta P) (z_0)=O\big(\big|P(z_0)-P_0\big| \eta+ \eta^{2}\big).
\end{equation}

\end{Lem}

\begin{proof}

Let  $G(x)=  \bigl\langle \nabla P(x), \tau\bigr\rangle.$  Then, similar to the estimate \eqref{luo-6},
by \eqref{eq-fhi} we have
\begin{equation}\label{00-luo-6}
\begin{split}
\int_{\mathbb R^3}G(x)U^{2}_{\eta,z_{0}}(x)=&
-2 \int_{\mathbb R^3} G(x)U_{\eta,z_{0}}(x) \varphi_{\eta, z_{0}}-\int_{\mathbb R^3} G(x)  \varphi^2_{\eta, z_{0}}+O\big(e^{-\frac{\theta}{\eta}}\big) \\=&
O\big( \big| P(z_0)-P_0\big|\eta^{5}+ \eta^{7} \big).
\end{split}
\end{equation}
On the other hand,  in view of $ G(z_0)=0,$ it is easy to show
\begin{equation}\label{01-luo-6}
\int_{\mathbb R^3} G(x) U^{2}_{\eta,z_{0}}(x)=\frac12
\eta^{5} \Delta G(z_0) B+O\bigl (\eta^{7}\bigr),
\end{equation}
where
\begin{equation}\label{luopeng10}
B=\frac{1}{3}\displaystyle\int_{\R^3}|x|^2U^2.
\end{equation}
Thus,  from  \eqref{00-luo-6} and \eqref{01-luo-6} we can obtain \eqref{01-1-1-8}.

\end{proof}

\begin{Thm}\label{Thm-1.a}
For $\lambda>0$ large, \eqref{1-23-5} has a solution $u_\lambda$ satisfying
\[
u_\lambda (x) =    \lambda\big(  U\bigl( \sqrt\lambda(x-x_{\lambda})\bigr) +\varpi_\lambda\big),
\]
where $x_{\lambda} \to b_0$ and $\int_{\mathbb R^3} \bigl(|\nabla \varpi_\lambda|^2 + \varpi_\lambda^2\bigr)\to 0$ as $\lambda\to +\infty.$
\end{Thm}

\begin{proof}
 As pointed out earlier, we need to solve \eqref{10-7-11}.   By Lemmas~\ref{lem-001-1-8} and
\ref{lem-002-1-8}, the equation \eqref{10-7-11} is equivalent to
\begin{equation*}
D_{  \nu}  P(z_0)=O\bigl( \eta^2 \bigr),\quad
(D_{\tau}  \Delta P) (z_0)=O\big(\big|P(z_0)-P_0\big|\eta+ \eta^{2}\big).
\end{equation*}
Let $\bar z_0\in \Gamma$ be the point such that $z_0-\bar z_0 = \alpha_0 \nu$ for
some $\alpha_0\in \mathbb R.$  Then, we have $D_{\nu} P(\bar z_0)=0.$
As a result,
\begin{eqnarray*}
D_{  \nu}  P(z_0)=D_{  \nu}  P(z_0)-D_{  \nu} P(\bar z_0)
=D^2_{ \nu \nu}  P(\bar z_0) \langle z_0-\bar z_0, \nu\rangle + O(|z_0-\bar z_0|^2).
\end{eqnarray*}
By the non-degenerate assumption,  
we find that $D_{  \nu}  P(z_0)=O\big(\eta^2\big)$ is equivalent to
$
\langle z_0-\bar z_0, \nu\rangle=O\bigl( \eta^2+ |z_0-\bar z_0|^2\bigr).$
This means that $
D_{\nu}  P(z_0)=O\bigl( \eta^2\bigr)$ can be written as
\begin{equation}\label{62-1-8}
|z_0-\bar z_0| =O\bigl( \eta^2\bigr).
\end{equation}
Let $\bar \tau_{j}$ be the $j$-th tangential unit vector of $\Gamma$ at $\bar z_0.$ Now by the assumption \textup{($P$)},
we have
\begin{eqnarray*}
(D_{  \tau_{j}}  \Delta P) (z_0)=(D_{ \bar  \tau_{j}}  \Delta P) (\bar z_0) +O(|z_0-\bar z_0|)=(D_{ \bar  \tau_{j}}  \Delta P) (\bar z_0)+O(\eta^2),
\end{eqnarray*}
and
\begin{equation*}
\begin{split}
(D_{\bar \tau_{j}}  \Delta P) (\bar z_0)
=&(D_{ \bar  \tau_{j}}  \Delta P) (\bar z_0)-(D_{  \tau_{j,0}}  \Delta P) (b_0)
=\bigl\langle (\nabla_T D_{  \tau_{j,0}}  \Delta P) ( b_0), \bar z_0-b_0\bigr\rangle
+ O(|\bar z_0-b_0|^2),
\end{split}
\end{equation*}
where $\nabla_{T}$ is the tangential gradient on
$\Gamma$ at  $b_0\in \Gamma,$ and  $\tau_{j, 0}$ is the $j$-th tangential unit vector of $\Gamma$ at $b_0.$
Therefore,  $(D_{\tau}  \Delta P) (z_0)=O\big( |P(z_0)-P_0| \eta+ \eta^{2}\big)$ can be rewritten as
\begin{equation}\label{63-1-8}
\bigl\langle (\nabla_T D_{  \tau_{j,0}}  \Delta P) ( b_0), \bar z_0-b_0\bigr\rangle= O({\eta^2}+|\bar z_0-b_0|^2).
\end{equation}
So we can solve \eqref{62-1-8} and \eqref{63-1-8} to obtain $z_0=x_{\eta,0}$ with  $x_{\eta,0} \to b_0$ as $ \eta\to 0.$
\end{proof}

Now we are in a position to prove Theorem \ref{nth1.2}.

\begin{proof}[\textbf{Proof of Theorem~\ref{nth1.2}}]

  Let $w_\lambda$ be a single-peak solution as in Theorem \ref{Thm-1.a}, and we define
\[
u_\lambda =\frac{w_\lambda}{\Bigl( \displaystyle\int_{\mathbb R^3}w_\lambda^2\Bigr)^{\frac12}}.
\]
Then $\displaystyle\int_{\mathbb R^3} u_\lambda^2=1$, and
\begin{equation*}\label{12-26-10}
-\Delta u_\lambda+  P(x)u_\lambda = \frac{a_\lambda}{8\pi} \int_{\R^{3}}\frac{u_\lambda^{2}(y)}{|x-y|}dyu_\lambda(x)-\lambda u_\lambda, \quad  \text{in}\; \mathbb R^3,
\end{equation*}
with
$
a_\lambda=\displaystyle\int_{\mathbb R^3} w_\lambda^2$.

Similar to \eqref{ab8-29-1}, we can prove
\begin{equation*}
\frac{1}{\sqrt{\lambda}} \int_{\mathbb R^3} w_\lambda^2=a_*+o\big(1\big), ~\mbox{as}~ a\rightarrow +\infty.
\end{equation*}
 Take $\lambda_0>0$ large and let $
a_0=\displaystyle\int_{\mathbb R^3} w_{\lambda_0}^2.$
 For any $a>0,$ let $f(\lambda)=\displaystyle\int_{\mathbb R^3} w_\lambda^2-a.$
Then for $a\geq a_{0},$  we have
$$f(\lambda_{0})=a_{0}-a\leq 0~\mbox{and}~\lim_{\lambda\rightarrow +\infty}f(\lambda)=\lim_{\lambda\rightarrow +\infty}\big(\sqrt{\lambda}(a_{*}+o(1))-a\big)=+\infty.$$
Hence by the continuity of the function $f(\lambda),$
 for any $a\geq a_{0},$
there exists $\lambda=\lambda_a>0$ large such that $f(\lambda_a)=0,$
i.e. $\displaystyle\int_{\mathbb R^3} w_\lambda^2= a,$
which yields that there exists $\lambda=\lambda_a>0$ large such that
 the solution $u_a$ of \eqref{1-23-5} with $\lambda=\lambda_a$ satisfies
$\displaystyle\int_{\mathbb R^3} w_\lambda^2= a.$
 Thus, for such $a$, we obtain a single-peak solution for \eqref{8-18-1}, where $\mu_a= -\lambda_a.$

\end{proof}

\section{Local uniqueness of single peak solutions}\label{s5}

From Lemma \ref{lem-5-05-1}, a single-peak solution $\tilde u_a$ to \eqref{30-7-11} can be written as
\begin{equation}\label{a-5-21-3}
\tilde{u}_a(x)= U_{\epsilon,x_{a}}+\varphi_{a}(x),
\end{equation}
with
$|x_{a}-b_0|=o(1),$
$\epsilon =\frac{1}{ \sqrt{-\mu_a}}$, $\varphi_a \in \displaystyle E_{a,{x}_{a}}$ and
\begin{equation}\label{06-09-2}
\|\varphi_{a}\|_a=O\Big( \displaystyle\big|P(x_{a})-P_0\big|\epsilon^{\frac{7}{2}} + \displaystyle\big|\nabla P(x_{a})\big|\epsilon^{\frac{9}{2}}
+\epsilon^{\frac{11}{2}}\Big).
\end{equation}
Also we know  $x_{a}\in \Gamma_{t_a}$ for some $t_a\to P_0.$
Similar to the last section, we use $\nu_{a}$ to denote the unit normal vector of $\Gamma_{t_a}$ at $x_{a},$ while we use $\tau_{a,j}$
 to denote the principal direction  of $\Gamma_{t_a}$ at $x_{a}.$ Then, at $x_{a},$ it holds
\begin{equation}\label{06-09-3}
 D_{\tau_{a,j}}  P(x_{a})=0,\quad \big|\nabla  P(x_{a})\big|= \big|D_{  \nu_{a}}  P(x_{a})\big|.
\end{equation}

 We first prove the following result.

\begin{Lem}
Under the assumption \textup{($P$)}, we have
\begin{equation}\label{1-1-8}
D_{  \nu_{a}}  P(x_{a})=O\bigl(\epsilon^2\bigr).
\end{equation}

\end{Lem}

\begin{proof}

 We use \eqref{30-31-7} to obtain
\begin{equation}\label{3-1-8}
\int_{B_{\rho}(x_{a})} D_{\nu_{a}}  P(x) \tilde{u}^2_a
=O(e^{-\frac{\theta}{\epsilon}}).
\end{equation}
Then by \eqref{a-5-21-3}--\eqref{06-09-3} and \eqref{3-1-8}, we get
\begin{equation}\label{aluo-6}
\begin{split}
\int_{B_{\rho}(x_{a})} & D_{\nu_{a}}  P(x)
U^2_{\epsilon,x_{a}}\\=&
-2 \int_{B_{\rho}(x_{a})} D_{  \nu_{a}}  P(x)
U_{\epsilon,x_{a}} \varphi_{a}-\int_{B_{\rho}(x_{a})}D_{\nu_{a}} P(x)
 \varphi^2_{a}+O\big(e^{-\frac{\theta}{\epsilon}}\big)\\=&
O\big(|  D_{  \nu_{a}}  P( x_{a})|\epsilon^{\frac{3}{2}} \cdot\|\varphi_{a}\|_a+\epsilon^{\frac{5}{2}}\|\varphi_{a}\|_a+
\|\varphi_{a}\|^2_a\big)+O\big(e^{-\frac{\theta}{\epsilon}}\big)\\
=&
O\big(\big|P( x_{a})-P_0\big|\epsilon^{5}+ |D_{  \nu_{a}}  P( x_{a})|\epsilon^{6}+\epsilon^{7}\big).
\end{split}
\end{equation}
On the other hand, by Taylor's expansion,
we have
\begin{equation}\label{abcluo-3}
\begin{aligned}
\int_{B_{\rho}(x_{a})}& D_{  \nu_{a}}  P(x)
U^2_{\epsilon,x_{a}}=\epsilon^{3} \Big[a_* D_{  \nu_{a}}  P( x_{a})
  +\frac{B\epsilon^{2}}{2} \Delta  D_{  \nu_{a}}  P(x_{a}) +O\big(\epsilon^{4}\big)\Big],
\end{aligned}
\end{equation}
where $B$ is the constant in \eqref{luopeng10}.
And then  \eqref{1-1-8}  follows from \eqref{aluo-6} and \eqref{abcluo-3}.
\end{proof}

Let $\bar  x_{a}\in \Gamma$ be the point such that $ x_{a}-\bar  x_{a} = \beta_{a} \nu_{a}$ for some $\beta_{a}\in \mathbb R.$
Then we can prove

\begin{Lem}
If the assumption \textup{($P$)} holds,  then we have
\begin{equation}\label{8-22-3}
\begin{cases}
 \bar x_{a}-b_0= L \epsilon^2 +O(\epsilon^{4}),\vspace{2mm}\\
 x_{a}-\bar  x_{a} = -\displaystyle\frac{B}{2a_*}\frac{\partial  \Delta P(b_0)}{\partial \nu} \Big(\frac{\partial^2 P(b_0)}{\partial \nu^2 } \Big)^{-1} \epsilon^{2} +
 O(\epsilon^{4}),
\end{cases}
\end{equation}
where $B$  is the constant in \eqref{luopeng10} and  $L$ is a  vector  depending  on $b_0.$
\end{Lem}

\begin{proof}

  It follows from  \eqref{aluo-6} and \eqref{abcluo-3} that
\begin{equation}\label{bluo-3}
\begin{split}
 &\big(a_*+O(\epsilon^{2})\big)  D_{  \nu_{a}}  P( x_{a})
  +\frac{B\epsilon^2}{2}\Delta  D_{  \nu_{a}}  P(x_{a})
  \\&=O\big(\epsilon^4+\epsilon^2
 |P(x_{a})-P_0|\big)=O\big(\epsilon^4+\epsilon^2 |x_{a}-\bar x_{a}|^2\big).
  \end{split}
\end{equation}
Since $\frac{\partial^2P(b_0)}{\partial \nu^2}\neq 0,$ the outward unit normal vector $\nu_{a}(x)$ and
the tangential unit vector $\tau_{a}(x)$ of $\Gamma_{t_a}$  at
$x_{a}$ are Lip-continuous   in $W_{\delta},$ from \eqref{bluo-3}, we find
\begin{equation}\label{2-13-8}
x_{a}-\bar  x_{a}
=-\frac{B}{2a_*}\big( \Delta  D_{  \nu}  P (b_0) \big) \Big(\frac{\partial^2P(b_0)}{\partial \nu^2} \Big)^{-1} \epsilon^{2}+O\big(\epsilon^{4}+\epsilon^{2}
\big|\bar  x_{a}-b_0\big|^{2}\big).
\end{equation}
Then \eqref{06-09-2} and \eqref{2-13-8} implies \begin{equation}\label{06-09-4}
\|\varphi_{a}\|_a=O\Big(|x_{a}-\bar  x_{a}|^{2}\epsilon^{\frac{7}{2}}
+\epsilon^{\frac{11}{2}}\Big)=O\big(\epsilon^{\frac{11}{2}}\big).
\end{equation}
Recall that $G(x)=  \bigl\langle \nabla P(x), \tau_{a}\bigr\rangle.$ Then
$
G(x_{a}) =0.
$
Similar to \eqref{luo-6} and \eqref{aluo-6},  we have
\begin{equation}\label{06-09-5}
\begin{split}
\int_{B_{\rho}(x_{a})}G(x)U^2_{\epsilon,x_{a}}
=&
-2 \int_{B_{\rho}(x_{a,i})} G(x) U_{\epsilon,x_{a}}\varphi_{a}-\int_{B_{\rho}(x_{a})} G(x)
\varphi^2_{a}+O\big(e^{-\frac{\theta}{\epsilon}}\big)\\
=&
-2 \int_{B_{\rho}(x_{a,i})} G(x) U_{\epsilon,x_{a}} \varphi_{a}
+O\bigl(\|\varphi_{a}\|^2_a\big)+O\big(e^{-\frac{\theta}{\epsilon}}\big)\\
=&-2 \int_{B_{\rho}(x_{a})} \langle \nabla G(x_{a}), x-x_{a}\rangle U_{\epsilon,x_{a}} \varphi_{a}
+O\big(\epsilon^{9}\big).
\end{split}
\end{equation}
On the other hand,  in view of
$\nabla P(x)=0,$  $ x\in \Gamma,$
we find
\begin{eqnarray}\label{06-09-6}
 \nabla G(x_{a})=\bigl\langle \nabla^2 P(x_{a}), \tau_{a}\bigr\rangle
 = \bigl\langle \nabla^2 P(\bar x_{a}), \bar \tau_{a}\bigr\rangle+
 O\big(|x_{a}-\bar x_{a}|\big)=  O\big(|x_{a}-\bar x_{a}|\big),
\end{eqnarray}
where  $\bar  x_{a}\in \Gamma$ is the point such that $ x_{a}-\bar x_{a} = \beta_{a} \nu_{a}$ for some $\beta_{a}\in \mathbb R$, and $\bar \tau_{a,j}$ is the tangential vector of $\Gamma$ at $\bar x_{a}\in \Gamma.$
Therefore, from \eqref{2-13-8}, \eqref{06-09-4} and \eqref{06-09-6}, we know
 \begin{equation}\label{06-09-7}
  \begin{split}
  \int_{B_{\rho}(x_{a,i})}& \langle \nabla G(x_{a}), x-x_{a}\rangle  U_{\epsilon,x_{a}} \varphi_{a}\\
  =&O\bigl( \epsilon^{\frac{5}{2}}  |\nabla G(x_{a})| \|\varphi_{a}\|_a\bigr)
  = O\big(|x_{a}-\bar x_{a}| \epsilon^{8}\big)=O\big(\epsilon^{10}\big).
  \end{split}
  \end{equation}
Then by \eqref{06-09-5} and  \eqref{06-09-7}, we find
\begin{equation}\label{06-09-8}
\begin{split}
\int_{B_{\rho}(x_{a})}G(x)U^{2}_{\epsilon,x_{a}}=O\big(\epsilon^{9}\big).
\end{split}
\end{equation}
On the other hand, by the Taylor's expansion,
we can prove
\begin{equation}\label{06-09-9}
\begin{aligned}
\int_{B_{d}(x_{a})}G(x)U^{2}_{\epsilon,x_{a}}=
\big[\frac{B\epsilon^{5}}{2} (1+P_0\epsilon^2)^{\frac{1}{2}}\big](D_{\tau_{a}} \Delta P)(x_{a})
+\frac{ H_{\tau} \epsilon^{7}}{24}+O\big(\epsilon^{9}\big),
\end{aligned}
\end{equation}
where
$$
H_{\tau_i} =\sum^2_{l=1}\sum^2_{m=1} \frac{\partial^4 G(b_0)}{ \partial x^2_l \partial x^2_m}\int_{\R^N}x_l^2x^2_mU^2.
$$
So \eqref{06-09-8} and \eqref{06-09-9} give
\begin{equation}\label{06-09-10}
\begin{aligned}
  (D_{\tau_{a}} \Delta P)(x_{a})
=-\frac{ H_{\tau} \epsilon^{2}}{1 2B}+O\big(\epsilon^{4}\big).
\end{aligned}
\end{equation}
We denote by $\bar \tau_{a}$ the tangential vector of $\Gamma$ at $\bar x_{a}.$ Then  by \eqref{2-13-8}, we get
\begin{equation*}\label{06-09-11}
\begin{split}
(D_{\tau_{a}} \Delta P)(x_{a})=&(D_{\bar \tau_{a}} \Delta P)(\bar x_{a})+ \langle A_{\tau}, x_{a}- \bar x_{a}\rangle+O(|
x_{a}- \bar x_{a}|^2)\\=&
(D_{\bar \tau_{a}} \Delta P)(\bar x_{a})+ B_{\tau} \epsilon^2 + O(\epsilon^4),
\end{split}
\end{equation*}
where $A_{\tau}$ is a vector depending on $b_0$  and $B_{\tau}$ is a constant  depending on $b_0.$ Moreover,
\begin{equation}\label{06-09-12}
(D_{\bar \tau_{a}} \Delta P)(\bar x_{a})= \Big(D^2_{\tau}(\Delta P)(b_0)\Big) (\bar x_{a}-b_0)+ O(|\bar x_{a}-b_0|^2).
\end{equation}
Therefore, from  \eqref{06-09-10}--\eqref{06-09-12}, we find
\begin{equation}\label{abc-lll}
\begin{split}
& D^2_{\tau_i}(\Delta P)(b_0) (\bar x_{a}-b_0)
=-\Big(\frac{ H_{\tau} }{1 2B}+B_{\tau}\Big)\epsilon^{2}+O\big(\epsilon^4\big)+ O(|\bar x_{a}-b_0|^2).
\end{split}
\end{equation}
Since  $D^2_{\tau} (\Delta P)(b_0) $ is non-singular, we can complete the proofs of \eqref{8-22-3} from \eqref{2-13-8} and \eqref{abc-lll}.
\end{proof}

Let
$$
\delta_a:=
\frac{a_*}{a}.
$$

\begin{Prop}
Under the assumption \textup{($P$)}, there holds
\begin{equation}\label{abc8-29-1}
-\mu_a\delta_a^2= 1+\gamma_1\delta_a^2+O\big(\delta_a^4\big),
\end{equation}
and
\begin{equation} \label{8-27-38}
x_{a}-b_0= \bar L \delta_a^2 +O(\delta_a^{4}),
\end{equation}
where $\gamma_1$ and the vector $\bar L$ are constants.
\end{Prop}

\begin{proof}

First, \eqref{ab8-29-1} shows that \eqref{abc8-29-1} holds.
Then we can find \eqref{8-27-38}  by \eqref{8-22-3} and \eqref{abc8-29-1}.
\end{proof}

 Let $u(x)\mapsto a^{-\frac{1}{2}}\delta^{-2}_{a} u(x).$ Then the problem
\eqref{8-18-1}--\eqref{8-28-3} can be changed into the following problem
\begin{equation}\label{1-12-11}
-\delta_a^2\Delta u+ \bigl( -\mu_a \delta_a^2 + \delta_a^2 P(x)\bigr)u=\frac{1}{8\pi\delta^{2}_{a}}\int_{\R^{3}} \frac{u^2(y)}{|x-y|}dy\,u(x),~
u\in H^1(\R^3),
\end{equation}
and
\begin{equation}\label{a1-12-11}
\int_{\R^3} u^2=a\delta_a^4.
\end{equation}
Then
similar to Lemma \ref{lem-5-05-1}, the single-peak solution of \eqref{1-12-11}--\eqref{a1-12-11}
concentrating at $b_0$ can be written as
$\tilde{U}_{\delta_a,x_{a}}+\tilde{\varphi}_{a}(x),$
with $
|x_{a}-b_0|=o(1),$ $\|\tilde{\varphi}_{a}\|_{\delta_a}=o(\delta_a^{\frac{3}{2}}),$
 and
\begin{equation*}
\begin{split}
{\tilde{\varphi}_{a}} \in \tilde{E}_{a,{x}_{a}}&:=\left \{ \varphi\in H^1(\R^3):
\Big\langle v,\frac{\partial \tilde{U}_{\delta_a,x_{a}}}{\partial{x_j}}\Big\rangle_{\delta_{a}}=0, ~j=1,2,3
 \right\},
 \end{split}
 \end{equation*}
where  $\tilde{U}_{\delta_a,x_{a}}:=\big(1+(\gamma_1+P_0)\delta_a^2\big)
U\Big(\frac{\sqrt{1+(\gamma_1+P_0)\delta_a^2}
(x-x_{a})}{\delta_a}\Big)$,  $\|\varphi\|^2_{\delta_a}:=\displaystyle\int_{\mathbb R^3} \bigl(  \delta_a^{2} |\nabla \varphi|^2 +\varphi^2\bigr)$
and $\gamma_1$ is the constant in \eqref{abc8-29-1}.
Then we can write the equation \eqref{1-12-11} as follows:
 \begin{equation*}
    \bar{L}_a(\tilde{\varphi}_{a})
    = \mathcal {R}_{a,\delta_{a}}\big(\tilde{\varphi}_{a}\big)+\bar{\mathcal {L}}_{a}(x),\end{equation*}
    where $\mathcal {R}_{a,\delta_{a}}$ is defined by \eqref{06-07-1},
 \begin{equation*}\label{06-07-3-2}
 \begin{split}
 \bar{L}_a(\tilde{\varphi}_{a}):=  & -\delta_{a}^{2}\Delta \tilde{\varphi}_{a}+\Big(\bigl( -\mu_a \delta_a^2 + \delta_a^2 P(x)\bigr)\tilde{\varphi}_{a}-
    \frac{1}{4\pi\delta_{a}^{2}}\int_{\R^{3}}
    \frac{\widetilde{U}_{\delta_{a},x_{a}}(y)\tilde{\varphi}_{a}(y)}{|x-y|}dy\widetilde{U}_{\delta_{a},x_{a}}(x)
    \\
    &\quad-\frac{1}{8\pi\delta_{a}^{2}}
    \int_{\R^{3}}\frac{(\widetilde{U}_{\delta_{a},x_{a}}(y))^{2}}{|x-y|}dy\tilde{\varphi}_{a}(x)\Big)
\end{split}
\end{equation*}
and
 \begin{equation*}\label{7-15-11}
 \begin{split}
  \tilde{\mathcal {L}}_{a}&=- \sum_{i=1}^m\big(-\mu_a \delta_a^2-(1+\gamma_{1}\delta_a^2)
  +(P(x)-P_{0})\delta_a^2\big)
  \tilde{U}_{\delta_a,x_{a}}.
 \end{split}
\end{equation*}

\begin{Lem}
There holds
 \begin{equation}\label{8-27-26}
\|\tilde{\varphi}_{a}\|_{\delta_a}=O\big(\delta_a^{\frac{11}{2}}\big).
\end{equation}
\end{Lem}

\begin{proof}
The proofs are similar to that of Lemma \ref{lem-5-05-1},
the difference is
 \begin{equation}\label{06-07-4}
   \|\tilde{\mathcal {L}}_{a}\|_{\delta_a}=O\Bigl( \big| P(x_{a})-P_0\big|\big)\delta_a^{\frac{7}{2}}
   +\big|\nabla P(x_{a})\big|\delta_a^{\frac{9}{2}}+ \delta_a^{\frac{11}{2}}\Bigr)
   =O\Bigl(\delta_a^{\frac{11}{2}}\Bigr).
 \end{equation}
Similar to Lemma \ref{lem-A.1},
 we can also check that  $\tilde{\mathcal {L}}_{a}$ is invertible in $\widetilde{E}_{a,x_{a}}.$
 Finally,  \eqref{06-07-4} and  the contradiction mapping theorem imply \eqref{8-27-26}.
\end{proof}

For simplicity of notations, hereafter we denote $\textbf{p}_{0}
:=1+(\gamma_1+P_0)\delta_a^2.$  Hence
$$
\tilde{U}_{\delta_a,x_{a}}:=\textbf{p}_{0}
U\Big(\frac{\sqrt{\textbf{p}_{0}}(x-x_{a})}{\delta_a}\Big).
$$

Let $u_a^{(1)}$ and $u_a^{(2)}$ be two single-peak solutions of \eqref{1-12-11}--\eqref{a1-12-11}
 concentrating  at some point $b_0,$ which can be written as
\begin{equation*}\label{8-20-1}
u_a^{(l)}=
\tilde U_{\delta_a,x_{a}^{(l)}}+\tilde{\varphi}^{(l)}_{a}(x),
~\mbox{for}~l=1,2,~\mbox{and}~\tilde{\varphi}^{(l)}_{a}\in    \tilde E_{a,{x}^{(l)}_{a}}.
\end{equation*}
Now we set
$
\xi_{a}(x)=\frac{u_{a}^{(1)}(x)-u_{a}^{(2)}(x)}
{\|u_{a}^{(1)}-u_{a}^{(2)}\|_{L^{\infty}(\R^3)}}.$
Then $\xi_{a}(x)$ satisfies $\|\xi_{a}\|_{L^{\infty}(\R^3)}=1.$ And from \eqref{1-12-11}, we find that
$\xi_a$ satisfies
\begin{equation*}
-\delta_a^2 \Delta \xi_{a}(x)+ C_{a}(x)\xi_{a}(x)-D_{a}(x)\xi_{a}(x)-E_{a}(x)=g_a(x),
\end{equation*}
where
\begin{equation*}
\begin{split}
&C_{a}(x)=\delta_a^2 P(x)-\delta_a^2 \mu_a^{(1)},~~D_{a}(x)=\frac{1}{8\pi \delta_a^2}\int_{\R^{3}}\frac{(u_a^{(1)})^{2}(y)}{|x-y|}dy,\,\,
\\
&E_{a}(x)=\frac{u_a^{(2)}(x)}{8\pi \delta_a^2}\int_{\R^{3}}\frac{(u_a^{(1)}+u_a^{(2)})\xi_{a}(y)}{|x-y|}dy,\,\,
g_{a}(x)=\frac{\delta_a^2 (\mu_{a}^{(1)}-\mu_{a}^{(2)})}
{\|u_{a}^{(1)}-u_{a}^{(2)}\|_{L^{\infty}(\R^3)}}u_a^{(2)}(x).
\end{split}
\end{equation*}
Also, similar to \eqref{2--5}, for any fixed $R\gg 1,$ there exist some $\theta>0$ and $C>0$, such that
\begin{equation}\label{a2---5}
|u^{(l)}_a(x)|+|\nabla u^{(l)}_a(x)|\leq Ce^{-\theta |x-x_{a}|/\delta_a},~\mbox{for}~l=1,2,\,\,
x\in \R^3\backslash B_{R \delta_a}(x_{a}).
\end{equation}
Now let $\bar \xi_{a}(x)=\xi_{a}\big(\frac{\delta_a }{\sqrt{\textbf{p}_{0}}}x+x^{(1)}_{a}\big),$
 we have
\begin{equation}\label{8-28-11}
-\Delta \bar \xi_{a}(x)+ \frac{C_{a}(\frac{\delta_a }{\sqrt{\textbf{p}_{0}}}x+x^{(1)}_{a})}{\textbf{p}_{0}}
\bar\xi_{a}(x)
-\frac{\bar{D}_{a}(x)}{\textbf{p}_{0}}\bar\xi_{a}(x)-\frac{\bar{E}_{a}(x)}{\textbf{p}_{0}}=
\frac{g_a(\frac{\delta_a }{\sqrt{\textbf{p}_{0}}}x+x^{(1)}_{a})}{\textbf{p}_{0}},
\end{equation}
where
\begin{equation*}
\begin{split}
&\bar{D}_{a}(x)=\frac{1}{8\pi \textbf{p}_{0}}
\int_{\R^{3}}\frac{(u_a^{(1)})^{2}(\frac{\delta_a }{\sqrt{\textbf{p}_{0}}}y+x^{(1)}_{a})}{|x-y|}dy,\,\,
\\
&\bar{E}_{a}(x)=\frac{u_a^{(2)}(\frac{\delta_a }{\sqrt{\textbf{p}_{0}}}x+x^{(1)}_{a})}{8\pi \textbf{p}_{0}}\int_{\R^{3}}\frac{(u_a^{(1)}+u_a^{(2)})(\frac{\delta_a }{\sqrt{\textbf{p}_{0}}}y+x^{(1)}_{a})\bar{\xi}_{a}(y)}{|x-y|}dy.
\end{split}
\end{equation*}

\begin{Lem}\label{lem-add1-7-11}
 For $x\in B_{\rho\sqrt{\textbf{p}_{0}}\delta_a^{-1}}(0),$ it holds
\begin{equation*}\label{add-eq1}
\frac{\bar{D}_{a}(x)}{\textbf{p}_{0}}=
\frac{1}{8\pi}\int_{\R^{3}}\frac{U^{2}(y)}{|x-y|}dy +O(\delta^{4}_{a}),
\end{equation*}
and
\begin{equation*}\label{add-eq2}
\begin{split}
\frac{\bar{E}_{a}(x)}{\textbf{p}_{0}}=\frac{U(x)}{4\pi }\int_{\R^{3}}\frac{U(y)\bar{\xi}_{a}(y)}{|x-y|}dy
+O\Big(\delta_a+\tilde{\varphi}^{2}_{a}(\frac{\delta_a x}{\sqrt{\textbf{p}_{0}}}+x^{(1)}_{a})\Big).
\end{split}
\end{equation*}

\end{Lem}

\begin{proof}
By direct computations, from \eqref{f} we have
\begin{equation*}
\begin{split}
\frac{\bar{D}_{a}(x)}{\textbf{p}_{0}}
&=\frac{1}{8\pi \textbf{p}^{2}_{0} }\int_{\R^{3}}\frac{(u_a^{(1)})^{2}(\frac{\delta_a }{\sqrt{\textbf{p}_{0}}}y+x^{(1)}_{a})}{|x-y|}dy\\
&=\frac{1}{8\pi  }\int_{\R^{3}}\frac{U^{2}(y) }{|x-y|}dy+\underbrace{\frac{1}{4\pi \textbf{p}_{0} }
\int_{\R^{3}}
\frac{U(y)\tilde{\varphi}^{(1)}_{a}(\frac{\delta_a y}{\sqrt{\textbf{p}_{i}}}+x^{(1)}_{a}) }{|x-y|}dy}_{:=F_{1}}\\
&\quad+\underbrace{\frac{1}{8\pi \textbf{p}^{2}_{0} }
\int_{\R^{3}}\frac{(\tilde{\varphi}^{(1)}_{a}(\frac{\delta_a y}{\sqrt{\textbf{p}_{i}}}+x^{(1)}_{a}) )^{2} }{|x-y|}dy}_{:=F_{2}}\\
&=\frac{1}{8\pi  }\int_{\R^{3}}\frac{U^{2}(x)}{|x-y|}dy+O(\delta^{4}_{a}).
\end{split}
\end{equation*}

Also, it follows from \eqref{g} that
\begin{equation*}
\begin{split}
\frac{\bar{E}_{a}(x)}{\textbf{p}_{0}}
&=\frac{u_a^{(2)}(\frac{\delta_a x}{\sqrt{\textbf{p}_{0}}}+x^{(1)}_{a})}{8\pi \textbf{p}^{2}_{0}}\int_{\R^{3}}
\frac{(u_a^{(1)}+u_a^{(2)})(\frac{\delta_a y}{\sqrt{\textbf{p}_{0}}}+x^{(1)}_{a})\bar{\xi}_{a}(y)}{|x-y|}dy\\
&=
\frac{U(x)}{4\pi }\int_{\R^{3}}\frac{U(y)\bar{\xi}_{a}(y)}{|x-y|}dy
\\
&\quad
+\underbrace{\frac{1}{4\pi \textbf{p}_{0}}\big(u_a^{(2)}(\frac{\delta_a x}{\sqrt{\textbf{p}_{0}}}+x^{(1)}_{a})
-\textbf{p}_{0}U(x)\big)\int_{\R^{3}}\frac{U(y)\bar{\xi}_{a}(y)}{|x-y|}dy}_{:=G_{1}}
\\
&\quad+\underbrace{\frac{u_a^{(2)}(\frac{\delta_a x}{\sqrt{\textbf{p}_{0}}}+x^{(1)}_{a})}{8\pi \textbf{p}^{2}_{0}}\int_{\R^{3}}
\frac{\big((u_a^{(1)}+u_a^{(2)})
(\frac{\delta_a y}{\sqrt{\textbf{p}_{0}}}+x^{(1)}_{a})-2\textbf{p}_{0}U(y)\big)\bar{\xi}_{a}(y)}{|x-y|}dy}_{:=G_{2}}
\\
&=\frac{U(x)}{4\pi }\int_{\R^{3}}\frac{U(y)\bar{\xi}_{a}(y)}{|x-y|}dy
+O\Big(\delta_a+\tilde{\varphi}^{2}_{a}(\frac{\delta_a x}{\sqrt{\textbf{p}_{0}}}+x^{(1)}_{a})\Big).
\end{split}
\end{equation*}

\end{proof}

\begin{Lem}\label{lem---1}
 For $x\in B_{\rho\sqrt{\textbf{p}_{0}}\delta_a^{-1}}(0),$ it holds
\begin{equation}\label{8-28-12}
\frac{C_{a}(\frac{\delta_a x}{\sqrt{\textbf{p}_{0}}}+x^{(1)}_{a})}{\textbf{p}_{0}}=
1 +O\Big(\delta_a^4+\sum^2_{l=1}
 \tilde{\varphi}_a^{(l)}(\frac{\delta_a x}{\sqrt{\textbf{p}_{0}}}+x^{(1)}_{a})
 \Big),
\end{equation}
and
\begin{equation}\label{8-28-13}
\begin{split}
\frac{g_a(\frac{\delta_a x}{\sqrt{\textbf{p}_{0}}}+x^{(1)}_{a})}{\textbf{p}_{0}}
=
-\frac{1}{4\pi a_*}U(x)\int_{\R^3}\int_{\R^3}\frac{U^2(x)U(y)\bar \xi_{a}(y)}{|x-y|}dx\,dy
+O\Big(\delta_a+\sum^2_{l=1}
 \tilde{\varphi}_a^{(l)}(\frac{\delta_a x}{\sqrt{\textbf{p}_{0}}}+x^{(1)}_{a})
 \Big).
\end{split}
\end{equation}
\end{Lem}

\begin{proof}
First, \eqref{8-28-12}  can be deduced by \eqref{abc8-29-1} and \eqref{8-27-38} directly.
 Now we prove \eqref{8-28-13}.

From \eqref{1-12-11} and \eqref{a1-12-11}, for $l=1,2$, we find
\begin{equation*}
\begin{split}
 a\mu^{(l)}_a \delta_a^6= \delta_a^2\int_{\R^3} \big( |\nabla u_a^{(l)}|^2+ P(x)(u_a^{(l)})^2\big)- \frac{1}{8\pi \delta^{2}_{a}}\int_{\R^3}\int_{\R^3}\frac{(u_a^{(l)})^2(x)(u_a^{(l)})^2(y)}{|x-y|}dx\,dy,
\end{split}\end{equation*}
which gives
\begin{equation}\label{06-10-2}
\begin{split}
&\frac{a\delta_a^6(\mu_{a}^{(1)}-\mu_{a}^{(2)})}
{\|u_{a}^{(1)}-u_{a}^{(2)}\|_{L^{\infty}(\R^3)}} \\=&
 -\mu_a^{(2)}\delta_a^{2}\int_{\R^3}
(u_a^{(1)}+u_a^{(2)}) \xi_a+\delta_a^2\int_{\R^3}  \big(  \nabla (u_a^{(1)}+u_a^{(2)})\cdot
 \nabla \xi_a + P(x)(u_a^{(1)}+u_a^{(2)}) \xi_a\big)\\&
- \frac{1}{8\pi \delta^{2}_{a}}\int_{\R^3}\int_{\R^3}\frac{(u_a^{(l)})^2(x)(u_a^{(1)}+u_a^{(2)})\xi_a(y)
+(u_a^{(1)}+u_a^{(2)})\xi_a(x)(u_a^{(2)})^2(y)}{|x-y|}dx\,dy
\\=&
-\bigl( \mu^{(2)}_a -\mu_a^{(1)}\bigr)\delta_a^2\int_{\R^3}  u_a^{(1)}\xi_a
- \frac{1}{8\pi \delta^{2}_{a}} \int_{\R^3}\int_{\R^3}
\frac{\big[(u_a^{(1)})^{2}(x)u_a^{(2)}(y)+(u_a^{(2)})^{2}(x)u_a^{(1)}(y)\big]\xi_a(y)}{|x-y|}dx\,dy,
\end{split}
\end{equation}
here we use the following identity:
\begin{equation*}\label{*}
\int_{\R^3}
 \bigl( u_a^{(1)}+ u_a^{(2)}\bigr)\xi_a= \frac{1}
{\|u_{a}^{(1)}-u_{a}^{(2)}\|_{L^{\infty}(\R^3)}}\Big(\int_{\R^3}
 ( u_a^{(1)} )^2 -\int_{\R^3}
 ( u_a^{(2)} )^2\Big) =0.
\end{equation*}
Then from \eqref{abc8-29-1}, \eqref{8-27-38}, \eqref{8-27-26} and \eqref{06-10-2}, we know
\begin{equation*}\label{06-10-1}
\begin{split}
\frac{\delta_a^2 (\mu_{a}^{(1)}-\mu_{a}^{(2)})}
{\|u_{a}^{(1)}-u_{a}^{(2)}\|_{L^{\infty}(\R^3)}}
&=\frac{1}{a_*\delta^{3}_a}
\bigl( \mu^{(2)}_a -\mu_a^{(1)}\bigr)\delta_a^2\int_{\R^3}  u_a^{(1)}\xi_a\\
&\quad
-\frac{1}{8 \pi a_*\delta^{5}_a} \int_{\R^3}\int_{\R^3}
\frac{\big[(u_a^{(1)})^{2}(x)u_a^{(2)}(y)+(u_a^{(2)})^{2}(x))u_a^{(1)}(y)\big]\xi_a(y)}{|x-y|}dx\,dy.
\end{split}
\end{equation*}

By \eqref{abc8-29-1}, H\"older inequality and \eqref{8-27-26}, we have
\begin{equation*}\label{20-7-11-2}
\begin{split}
&\frac{1}{a_*\delta^{3}_a}
\bigl( \mu^{(2)}_a -\mu_a^{(1)}\bigr)\delta_a^2\int_{\R^3}  u_a^{(1)}\xi_a\\
&=\frac{-1}{a_*\delta^{3}_a}
\bigl( \mu^{(2)}_a -\mu_a^{(1)}\bigr)\delta_a^2
\Big(\int_{\R^3}  \textbf{p}_{0}U\big(\frac{\sqrt{\textbf{p}_{0}}}{\delta_{a}}(x-x^{(1)}_{a}\big)\xi_a
+\int_{\R^3}\tilde{\varphi}_a^{(1)}\xi_a\Big)\\
&=\frac{1}{a_*\delta^{3}_a}O(\delta^{4}_{a})
\big(C\delta^{\frac{3}{2}}_{a}+C\delta^{\frac{11}{2}}_{a}\Big)
=O(\delta^{\frac{5}{2}}_a).
\end{split}
\end{equation*}

From \eqref{uu-1-1}, we can check that
\begin{equation*}\label{20-7-11-1}
\begin{split}
&\frac{1}{8 \pi a_*\delta^{5}_a}\int_{\R^3}\int_{\R^3}
\frac{\big[(u_a^{(1)})^{2}(x)u_a^{(2)}(y)+(u_a^{(2)})^{2}(x))u_a^{(1)}(y)\big]\xi_a(y)}{|x-y|}dx\,dy
\\
&=\frac{1}{8 \pi a_*\delta^{5}_a}\Big(2\delta_{a}^{5}\textbf{p}^{\frac{1}{2}}_{0}\int_{\R^{3}}\int_{\R^{3}}
\frac{U^{2}(x)U(y)\bar{\xi}_{a}(y)}{|x-y|}dxdy +O(\delta^{6}_{a}) \Big)
\\
&=\frac{1}{4 \pi a_*}\int_{\R^{3}}\int_{\R^{3}}
\frac{U^{2}(x)U(y)\bar{\xi}_{a}(y)}{|x-y|}dxdy +O(\delta^{6}_{a}).
\end{split}
\end{equation*}
Noting that in $x\in B_{\rho\sqrt{\textbf{p}_{0}}\delta_a^{-1}}(0),$
by the mean value theorem and \eqref{8-27-38} we have
\begin{equation*}\label{20-7-11-1}
\begin{split}
&u^{(2)}_a(\frac{\delta_a x}{\sqrt{\textbf{p}_{0}}}+x^{(1)}_{a})\\
&=\textbf{p}_{0}U(x)+\textbf{p}_{0}\Big(U\big(x+\frac{\sqrt{\textbf{p}_{0}}}{\delta_a}(x^{(1)}_{a}-x^{(2)}_{a}\big)-U(x)\Big)
+\tilde{\varphi}^{(2)}_{a}(\frac{\delta_a x}{\sqrt{\textbf{p}_{0}}}+x^{(1)}_{a})\\
&=\textbf{p}_{0}U(x)
+\textbf{p}_{0}\nabla U\big(\zeta x+(1-\zeta)\frac{\sqrt{\textbf{p}_{0}}}{\delta_a}(x^{(1)}_{a}-x^{(2)}_{a})\big)
\cdot \frac{\sqrt{\textbf{p}_{0}}}{\delta_a}(x^{(1)}_{a}-x^{(2)}_{a})
+\tilde{\varphi}^{(2)}_{a}\big(\frac{\delta_a x}{\sqrt{\textbf{p}_{0}}}+x^{(1)}_{a}\big)
\\
&=\textbf{p}_{0}U(x)+O(\delta_{a})+\tilde{\varphi}^{(2)}_{a}\big(\frac{\delta_a x}{\sqrt{\textbf{p}_{0}}}+x^{(1)}_{a}\big),
\end{split}
\end{equation*}
which combining all the estimates above implies that \eqref{8-28-13} holds.
\end{proof}

Then from Lemmas \ref{lem-add1-7-11} and \ref{lem---1}, we have the following result.
\begin{Lem}
  From $|\bar \xi_{a}|\le 1,$ we suppose that  $\bar \xi_{a}(x)\rightarrow \xi(x)$ in $C^1_{loc}(\R^3).$
   Then  $\xi(x)$ satisfies following system:
  \begin{equation*}
  \begin{split}
&-\Delta \xi(x)+\xi(x)
-\frac{1}{8\pi}\int_{\R^{3}}\frac{U^{2}(y)}{|x-y|}dy\xi(x)
-\frac{1}{4\pi}\int_{\R^{3}}\frac{U(y)\xi(y)}{|x-y|}dy\,U(x)
\\
&\quad=-\frac{1}{4\pi a_*}U(x)\int_{\R^3}\int_{\R^3}\frac{U^2(x)U(y)\xi(y)}{|x-y|}dx\,dy.
\end{split}
\end{equation*}
\end{Lem}

To prove $\xi=0,$
 we write
\begin{equation}\label{aaaaaluo--1}
\bar \xi _{a}(x)= \sum^3_{j=0}\beta_{a,j}\psi_j+\tilde{\xi}_{a}(x),~\mbox{in} ~H^1(\R^3),
\end{equation}
where $\psi_j (j=0,1,2,3)$ are the functions in \eqref{aaaaa} and $\tilde{\xi}_{a}(x)\in \tilde{E}$ with
\begin{equation*}
\tilde{E}=\big\{u\in H^1(\R^3), \langle u,\psi_j\rangle=0, ~\mbox{for}~j=0,1,2,3\big\}.
\end{equation*}
It is standard to prove the following result.

\begin{Lem}\label{lem8-27-4}
For any $u\in \tilde{E}$,  there exists $\bar\gamma>0$ such that
\begin{equation*}
\|\tilde  L(u)\|\geq \bar\gamma \|u\|,
\end{equation*}
where $\tilde  L$ is defined by
\begin{equation*}
\begin{split}
\tilde  L(u):&=-\Delta u(x)+u(x)
-\frac{1}{8\pi}\int_{\R^{3}}\frac{U^{2}(y)}{|x-y|}dyu(x)
-\frac{1}{4\pi}\int_{\R^{3}}\frac{U(y)u(y)}{|x-y|}dy\,U(x)
\\
&\quad+\frac{1}{4\pi a_*}U(x)\int_{\R^3}\int_{\R^3}\frac{U^2(x)U(y)u(y)}{|x-y|}dx\,dy.
\end{split}
\end{equation*}
\end{Lem}

\begin{Prop}
Let $\tilde{\xi}_{a}(x)$ be as in \eqref{aaaaaluo--1}. Then
\begin{equation}\label{8-27-2}
 \|\tilde{\xi}_{a}\| =O(\delta_a).
\end{equation}
\end{Prop}

\begin{proof}
First, Lemma \ref{lem8-27-4} gives
\begin{equation}\label{8-28-21}
\|\tilde{\xi}_{a}\|\le C \|\tilde  L(\tilde{\xi}_{a})\|.
\end{equation}
On the other hand,  from \eqref{8-28-11}--\eqref{aaaaaluo--1}, we can prove
\begin{equation}\label{8-28-26}
\begin{split}
 \tilde{L}(\tilde{\xi}_{a})&=
  \Big(1-\frac{C_{a}(\frac{\delta_{a}}{\sqrt{\textbf{p}_{0}}}x+x^{(1)}_{a})}{\textbf{p}_{0}}\Big)
 \bar{\xi}_{a}(x)
 +\Big(
 -\frac{\bar{D}_{a}(x)}{\textbf{p}_{0}}\bar{\xi}_{a}(x)-\frac{1}{8\pi}\int_{\R^{3}}\frac{U^{2}(y)}{|x-y|}dy\bar{\xi}_{a}(x)\Big)
 \\
 &
 \quad+
 \Big(-\frac{\bar{E}_{a}(x)}{\textbf{p}_{0}}\bar{\xi}_{a}(x)-\frac{1}{4\pi}\int_{\R^{3}}\frac{U(y)\bar{\xi}_{a}(y)}{|x-y|}dy\,U(x)\Big)
 \\
 &\quad
+\Big(\frac{g_{a}(\frac{\delta_{a}}{\sqrt{\textbf{p}_{0}}}x+x^{(1)}_{a})}{\textbf{p}_{0}}\bar{\xi}_{a}(x)
+\frac{1}{4\pi a_*}U(x)\int_{\R^3}\int_{\R^3}\frac{U^2(x)U(y)\bar{\xi}_{a}(y)}{|x-y|}dx\,dy\Big).
\end{split}
\end{equation}
So from \eqref{8-27-26}, \eqref{8-28-21}, \eqref{8-28-26}, Lemma \ref{lem-add1-7-11}
and Lemma \ref{lem---1}, we have for any $\upsilon(x)\in H^{1}(\R^{3})$
\begin{equation*}
\begin{split}
 \langle\tilde  L(\tilde{\xi}_{a}),\upsilon\rangle& = O\big(\delta_a\big)\|\upsilon\|
 +O\Big(\int_{\R^{3}}\sum^2_{l=1}
 \varphi_a^{(l)}(\delta_a x+x^{(1)}_{a,i})|\upsilon|
 \Big)\\
 &\leq O\big(\delta_a\big)\|\upsilon\|+C\delta_{a}^{-\frac{3}{2}}\| \varphi_a^{(l)}\|_{\delta_{a}}\|\upsilon\|\\
  &\leq O\big(\delta_a\big)\|\upsilon\|,
\end{split}
\end{equation*}
which implies \eqref{8-27-2}.

\end{proof}

Letting $\beta_{a,j}$ be as in \eqref{aaaaaluo--1} and using $| \bar \xi_{a}|\le 1$, we find
\begin{equation*}
\beta_{a,j}= \frac{
   \bigl\langle  \bar \xi_{a}, \psi_j\bigr\rangle}{\|\psi_j\|^2} =O\bigl(\|\bar \xi_{a}\|\bigr)= O(1),~ j=0,1,2,3.
\end{equation*}

\begin{Lem}\label{lem7-27}
There holds
\begin{equation}\label{8-28-44}
\beta_{a,0}=o(1).
\end{equation}
\end{Lem}

\begin{proof}
On one hand, from  \eqref{aaaaaluo--1}, \eqref{8-27-2}
and \eqref{8-27-38} we get
 \begin{equation}\label{20-7-21-1}
\begin{split}
&\int_{B_\rho(x^{(1)}_{a})} (u_a^{(1)}+u_a^{(2)}) \xi_{a}\\
&=\frac{\delta^{3}_{a}}{\textbf{p}^{\frac{3}{2}}_{0}}
\int_{B_{\frac{\sqrt{\textbf{p}_{0}}}{\delta_{a}}\rho(0)}}
\Big(2\textbf{p}_{0}U(x)+\nabla U\big(\zeta x+(1-\zeta)\frac{\sqrt{\textbf{p}_{0}}}{\delta_{a}}(x^{(1)}_{a}-x^{(2)}_{a})\big)
\cdot \frac{\sqrt{\textbf{p}_{0}}}{\delta_{a}}(x^{(1)}_{a}-x^{(2)}_{a})
\\
&\quad\quad\quad\quad\quad\quad\quad\quad
+\sum_{l=1}^{2}\tilde{\varphi}^{(l)}(\frac{\sqrt{\textbf{p}_{0}}}{\delta_{a}}x+x^{(1)}_{a})\Big) \bar{\xi}_{a}(x)dx\\
&=2\frac{\delta^{3}_{a}}{\sqrt{\textbf{p}_{0}}}\int_{\R^{3}}U(x)\bar{\xi}_{a}(x)
-\int_{\R^{3}\backslash B_{\frac{\sqrt{\textbf{p}_{0}}}{\delta_{a}}\rho(0)}}U(x)\bar{\xi}_{a}(x)
\\
&\quad+O\big(\delta^{2}_{a}|x^{(1)}_{a}-x^{(2)}_{a}|\big)\Big|\nabla U\big(\zeta x+(1-\zeta)\frac{\sqrt{\textbf{p}_{0}}}{\delta_{a}}(x^{(1)}_{a}-x^{(2)}_{a})\big)\Big|_{2}|\bar{\xi}_{a}(x)|_{2}
\\
&\quad+\frac{\delta^{3}_{a}}{\textbf{p}^{\frac{3}{2}}_{0}}
\sum_{l=1}^{2}\Big|\tilde{\varphi}^{(l)}(\frac{\sqrt{\textbf{p}_{0}}}{\delta_{a}}x+x^{(1)}_{a})\Big|_{2}|\bar{\xi}_{a}(x)|_{2}\\
&=2\frac{\delta^{3}_{a}}{\sqrt{\textbf{p}_{0}}}\Big(\int_{\R^{3}}U(x)\beta_{a,0}(2U(x)+x\cdot \nabla U(x))
+\int_{\R^{3}}U(x)\tilde{\xi}_{a}(x)\Big)+O\big(e^{-\frac{\theta}{\delta_{a}}}\big)
\\
&\quad+O\big(|x^{(1)}_a-x^{(2)}_a|\delta_a^{2}\big)
+O\Big(\delta_a^{\frac{3}{2}} \sum_{l=1}^{2}\|\tilde{\varphi}_a^{(l)}\|_{\delta_a}\Big)\\
&=\beta_{a,0}\frac{\delta^{3}_a }{\sqrt{\textbf{p}_{0}}}\int_{\R^{3}}U^{2}(x)+O(\delta_{a}^{4})
+O\big(|x^{(1)}_a-x^{(2)}_a|\delta_a^{2}
+O(\delta^{4}_a)+\delta_a^{\frac{3}{2}} \|\tilde{\varphi}_a^{(2)}\|_{\delta_a}\big)\\
&=
a_*\beta_{a,0}\delta^{3}_a +
 O\big(\delta_a^{4}\big).
 \end{split}
\end{equation}
On the other hand,
noting that from \eqref{a2---5}
\begin{equation*}\label{20-7-21-2}
\begin{split}
\int_{\R^{3}\setminus B_{\rho}(x_{a}^{(1)})}(u_{a}^{(1)}+u_{a}^{(2)})\xi_{a}
&=O\Big(\int_{\R^{3}\setminus B_{\rho}(x_{a}^{(1)})}|u_{a}^{(1)}+u_{a}^{(2)}|\Big)\\
&=O\Big(\int_{\R^{3}\setminus B_{\rho}(x_{a}^{(1)})}|u_{a}^{(1)}|\Big)
+O\Big(\int_{\R^{3}\setminus B_{\frac{\rho}{2}}(x_{a}^{(2)})}|u_{a}^{(2)}|\Big)\\
&=O\big(e^{-\frac{\theta}{\delta_{a}}}\big),
\end{split}
\end{equation*}

then we can have
\begin{equation}\label{20-7-21-2}
\begin{split}
\int_{B_{\rho}(x_{a}^{(1)})}(u_{a}^{(1)}+u_{a}^{(2)})\xi_{a}
&=\int_{\R^{3}}(u_{a}^{(1)}+u_{a}^{(2)})\xi_{a}
+O\big(e^{-\frac{\theta}{\delta_{a}}}\big)\\
&=\frac{1}{\|u_{a}^{(1)}-u_{a}^{(2)}\|_{L^{\infty}(\R^{3})}}\int_{\R^{3}}\big[(u_{a}^{(1)})^{2}-(u_{a}^{(2)})^{2}\big]
+O\big(e^{-\frac{\theta}{\delta_{a}}}\big)\\
&=O\big(e^{-\frac{\theta}{\delta_{a}}}\big),
\end{split}
\end{equation}
since
$$
\int_{\R^{3}}(u_{a}^{(1)})^{2}=\int_{\R^{3}}(u_{a}^{(2)})^{2}=1.
$$
It follows form \eqref{20-7-21-1} and \eqref{20-7-21-2} that \eqref{8-28-44} holds.
\end{proof}

\begin{Rem}
We would like to point out that
since here we consider the single-peak case and
observing this fact $\displaystyle\int_{\R^{3}}U(x)(2U(x)+x\cdot \nabla U(x))dx=\frac{1}{2}\displaystyle\int_{\R^{3}}U^{2}(x)dx\neq 0,$
 the proof of Proposition \ref{lem7-27} is much simple.
\end{Rem}

\begin{Prop}
It holds
\begin{equation}\label{add8-28-44}
\beta_{a,j}=o(1),~~j=1,2,3.
\end{equation}
\end{Prop}

\begin{proof}
\textbf{Step 1:} To prove $\beta_{a,3}=O(\delta_a).$

\smallskip

Using \eqref{30-31-7}, we deduce
\begin{equation}\label{3-18}
\displaystyle \int_{B_{\rho}(x_{a}^{(1)})}\frac{\partial P(x)}{\partial  \nu_{a}}B_a(x)\xi_{a}=
O\big(e^{-\frac{\theta}{\delta_a}}\big),
\end{equation}
where $\nu_{a}$ is the outward unit vector of $\partial B_{d}(x_{a}^{(1)})$ at $x$, $
B_a(x)=\displaystyle\sum^2_{l=1}u_{a}^{(l)}(x).$

On the other hand, by \eqref{8-27-38}, we have
\begin{equation}\label{0-16-8}
\begin{split}
B_a(x)&=\Big(2 +O(\delta^{2}_a)\Big)
\tilde{U}_{\delta_a,x_{a}^{(1)}}(x)
+O\Big(\sum^{2}_{l=1}|\tilde{\varphi}_{a}^{(l)}(x)|\Big),\,\,\,x\in B_{\rho}(x_{a}^{(1)}).
\end{split}
\end{equation}
Also, from \eqref{8-22-3}, we find
\begin{eqnarray*}
\frac{\partial P(x_{a}^{(1)})}{\partial  \nu_{a}}=\frac{\partial P(x_{a}^{(1)})}{\partial  \nu_{a}}-
\frac{\partial P(\bar x_{a}^{(1)})}{\partial  \nu_{a}}= O\big(\big|x_{a}^{(1)}-\bar x_{a}^{(1)}\big|\big)=O(\delta_a^2),
\end{eqnarray*}
and
\begin{eqnarray*}
\frac{\partial^2 P(x_{a}^{(1)})}{\partial  \nu_{a}\partial  \tau_{a,j}}=\frac{\partial^2 P(x_{a}^{(1)})}{\partial  \nu_{a}\partial  \tau_{a,j}}-
\frac{\partial^2 P(\bar x_{a}^{(1)})}{\partial  \nu_{a}\partial  \tau_{a,j}}= O\big(\big|x_{a}^{(1)}-\bar x_{a}^{(1)}\big|\big)=O(\delta_a^2),~\mbox{for}~j=1,2.
\end{eqnarray*}
From \eqref{2--5}, \eqref{30-31-7} and \eqref{0-16-8},  we get
\begin{equation}\label{3--18}
\begin{split}
 &\displaystyle \int_{B_{\rho}(x_{a}^{(1)})}\frac{\partial P(x)}{\partial  \nu_{a}}B_a(x)\xi_{a}
 \\&=
\int_{\R^3}  \frac{\partial P(x_{a}^{(1)})}{\partial  \nu_{a}} B_a(x)\xi_{a}+\int_{\R^3} \Bigl\langle \nabla \frac{\partial P(x_{a}^{(1)})}{\partial  \nu_{a}}, x-x_{a}^{(1)}\Bigr\rangle B_a(x)\xi_{a}
 +O\big(\delta_a^{5}\big)\\
& =
-\frac{\partial^2 P(x_{a}^{(1)})}{\partial  \nu^2_{a}} a_*\beta_{a,3} \delta_a^{4}+O\big(\delta_a^{5}\big).
\end{split}
\end{equation}
Then \eqref{3-18} and \eqref{3--18} imply  $\beta_{a,3}=O(\delta_a).$

\medskip

\noindent\textbf{Step 2:} To prove  $\beta_{a,j}=o(1)$ for $j=1,2.$

\smallskip

Similar to \eqref{3-18}, we have
\begin{equation}\label{3.-14}
 \int_{B_{\rho}(x_{a}^{(1)})} \frac{\partial P(y)}{\partial \tau_{a,j}}B_a(y)\xi_{a}=O\big(e^{-\frac{\theta}{\delta_a}}\big),
 ~\mbox{for}~j=1,2.
 \end{equation}
Using suitable rotation, we assume that $\tau_{a,1} =(1, 0,0),\tau_{a, 2} =(0,1,0)$ and $\nu_{a} =(0,0,1).$
Under the assumption \textup{($P$)}, we obtain
\begin{equation}\label{1-16-8}
\begin{split}
 &\frac{\partial P(\frac{\delta_a}{\sqrt{\textbf{p}_{0}}} y+ x^{(1)}_{a})}{\partial \tau_{a,j}}\\
 =&
\frac{\delta_a}{\sqrt{\textbf{p}_{0}}}\sum^3_{l=1}\frac{\partial^2 P(x^{(1)}_{a})}{\partial y_l \partial \tau_{a,j} } y_l+\frac{\delta_a^2}{2\textbf{p}_{0}}
\sum^3_{k=1}\sum^3_{l=1}\frac{\partial^3 P(x^{(1)}_{a})}{\partial y_k  \partial y_l \partial \tau_{a,j}} y_k y_l\\&+
\frac{\delta^3_a}{6\textbf{p}^{\frac{3}{2}}_{0}}\sum^3_{s=1}\sum^3_{k=1}
\sum^3_{l=1}\frac{\partial^4 P(x^{(1)}_{a})}{\partial y_s \partial y_l \partial y_k \partial \tau_{a,i,j}} y_s y_ly_k
+o\big(\delta_a^3|y|^3\big),~\mbox{in}~B_{\frac{\rho}{\delta_a}\sqrt{\textbf{p}_{0}}}(0).
\end{split}\end{equation}
By \eqref{1.6}, \eqref{8-22-3}, \eqref{8-27-26}, \eqref{aaaaaluo--1}, \eqref{0-16-8} and
the symmetry of $\psi_j(x),$ we find, for $j=1,2,$
\begin{equation*}\label{2-16-8}
\begin{split}
\sum^3_{k=1}\sum^3_{l=1}&\frac{\partial^3 P(x^{(1)}_{a})}{\partial y_k  \partial y_l \partial \tau_{a,j}}
\int_{B_{\frac{\rho}{\delta_{a}}\sqrt{\textbf{p}_{0}}}(0)}
B_a(\frac{\delta_a}{\sqrt{\textbf{p}_{0}}} y+ x^{(1)}_{a}) \bar \xi_{a} y_k y_l\\=&
2 \sum^3_{k=1}\sum^3_{l=1}\frac{\partial^3 P(x^{(1)}_{a})}{\partial y_k  \partial y_l \partial \tau_{a,j}}
\int_{B_{\frac{\rho}{\delta_{a}}\sqrt{\textbf{p}_{0}}}(0)} U\big(y\big)\bar \xi_{a,i} y_k y_l+O(\delta_a^2)
\\
=& B\beta_{a,0} \frac{\partial \Delta P(x^{(1)}_{a})}{\partial \tau_{a,j}} +O(\delta_a^2)=
O\big(|x^{(1)}_{a}-b_0|\big) +O(\delta_a^2)=O(\delta_a^2).
\end{split}
\end{equation*}
Also from \eqref{add8-28-44} and \eqref{0-16-8}, we get
\begin{equation*}\label{3-16-8}
\begin{split}
& \sum^3_{s=1}\sum^3_{k=1}\sum^3_{l=1}\frac{\partial^4 P(x^{(1)}_{a})}{\partial y_s \partial y_l \partial y_k \partial \tau_{a,j}}
 \int_{B_{\rho\sqrt{\textbf{p}_{0}}\delta_a^{-1}}(0)} B_a(\frac{\delta_a}{\sqrt{\textbf{p}_{0}}} y+ x^{(1)}_{a}) \bar \xi_{a}y_s y_ly_m
\\=&
2 \sum^3_{s=1}\sum^3_{k=1}\sum^3_{l=1}\frac{\partial^4 P(x^{(1)}_{a})}{\partial y_s \partial y_l \partial y_k \partial \tau_{a,j}}
 \int_{B_{\rho\sqrt{\textbf{p}_{0}}\delta_a^{-1}}(0)}  U(y)(\sum^{2}_{h=1}\beta_{a,h})\psi_h(y)y_s y_ly_k
+o\big(1\big)\\=&
2 \sum^{2}_{h=1}\beta_{a,j}
 \int_{B_{\rho\sqrt{\textbf{p}_{0}}\delta_a^{-1}}(0)} U(y) \psi_h(y)y_h \Big[\frac{\partial^4 P(x^{(1)}_{a})}{\partial y^3_h \partial \tau_{a,j}}y_h^2
 +3 \sum^3_{l=1,l\neq h}\frac{\partial^4 P(x^{(1)}_{a})}{\partial y_h \partial y^2_l \partial \tau_{a, j}} y_l^2\Big]+o\big(1\big)\\=&
-3B\Big(\sum_{h=1}^{2}\frac{\partial^2 \Delta P(x^{(1)}_{a})}{\partial \tau_{a,h} \partial \tau_{a,j}}  \beta_{a,h}\Big)+o\big(1\big)=
-3B\Big(\sum_{h=1}^{2}\frac{\partial^2 \Delta P(b_0)}{\partial \tau_{a,h} \partial \tau_{a,j}}  \beta_{a,h}\Big)+o\big(1\big).
\end{split}
\end{equation*}
By \eqref{ab7-19-29},  we estimate
\[
\frac{\partial^2 P(x^{(1)}_{a})}{\partial y_l \partial \tau_{a, j} }=  - \frac{\partial P(x^{(1)}_{a})}{ \partial \nu_{a} }
\kappa_{l}(x^{(1)}_{a})\delta_{lj},\quad l,j=1,2.
\]
 Since  $\frac{\partial P(\bar x^{(1)}_{a})}{ \partial \nu_{a} }=0,$ from \eqref{8-22-3}, we find
\begin{equation}\label{lll}
\begin{split}
\frac{\partial^2 P(x^{(1)}_{a})}{\partial y_l \partial \tau_{a,j} }=&  -\Bigl(  \frac{\partial P(x^{(1)}_{a})}{ \partial \nu_{a} }-
\frac{\partial P(\bar x^{(1)}_{a})}{ \partial \nu_{a} }
\Bigr)
\kappa_{l} (x^{(1)}_{a})\delta_{lj} \\
=&- \frac{\partial^2 P(\bar x^{(1)}_{a})}{ \partial \nu^2_{a} } ( x^{(1)}_{a}-\bar x^{(1)}_{a})\cdot \nu_{a}\kappa_{l} (x^{(1)}_{a})\delta_{lj}+o(\delta_a^2)\\
=&- \frac{\partial^2 P(b_0)}{ \partial \nu^2 } ( x^{(1)}_{a}-\bar x^{(1)}_{a})\cdot \nu_{a}\kappa_{l} (b_0)\delta_{lj}+o(\delta_a^2)\\
=&\frac{B}{2a_*}\frac{\partial  \Delta P(b_0)}{\partial \nu}  \delta_a^{2}\kappa_{l}(b_0)\delta_{lj}+o(\delta_a^2).
\end{split}
\end{equation}
Therefore from \eqref{8-27-26}, \eqref{aaaaaluo--1}, \eqref{0-16-8} and \eqref{lll}, we get
\begin{equation}\label{5-16-8}
\begin{split}
\sum^3_{l=1} &\frac{\partial^2 P(x^{(1)}_{a})}{\partial y_l \partial \tau_{a,j} }\int_{B_{\frac{\rho}{\delta_{a}}\sqrt{\textbf{p}_{0}}}(0)}
 B_a(\frac{\delta_a}{\sqrt{\textbf{p}_{0}}} y+ x^{(1)}_{a}) \bar \xi_{a} y_l\\
=& \frac{B}{2a_*}\frac{\partial  \Delta P(b_0)}{\partial \nu}  \delta_a^{2}\kappa_{j} (b_0)
\int_{B_{\frac{\rho}{\delta_{a}}\sqrt{\textbf{p}_{0}}}(0)}
B_a(\frac{\delta_a}{\sqrt{\textbf{p}_{0}}} y+ x^{(1)}_{a}) \bar \xi_{a} y_j
+o(\delta_a^2)\\
=& \frac{B}{a_*}\frac{\partial  \Delta P(b_0)}{\partial \nu_i}  \delta_a^{2}\kappa_{j}(b_0)\beta_{a,j} \int_{\mathbb R^3 }U(y)
 \frac{U(y)}{\partial y_j} y_j
+o(\delta_a^2)\\
=& -\frac{B}{2}\frac{\partial  \Delta P(b_0)}{\partial \nu_i}  \delta_a^{2}\kappa_{j}(b_0)\beta_{a,j}
+o(\delta_a^2).
\end{split}\end{equation}
Combining \eqref{1-16-8} to \eqref{5-16-8}, we obtain
\begin{equation}\label{6-16-8}
\begin{split}
\int_{B_{\rho}(x_{a}^{(1)})} & \frac{\partial P(y)}{\partial \tau_{a,j}}B_a(y)\xi_{a}\\
=&-\frac{B}{2}\frac{\partial  \Delta P(b_0)}{\partial \nu}  \delta_a^{6}\kappa_{j}(b_0)\beta_{a,j}
-\frac{B}{2}\Big(\sum_{l=1}^{2}\frac{\partial^2 \Delta P(b_0)}{\partial \tau_{l} \partial \tau_{ j}} \beta_{a,j} \Big) \delta_a^{6}+o(\delta_a^{6}).
 \end{split}\end{equation}
 From \eqref{3.-14} and \eqref{6-16-8}, we find
\[
\frac{\partial  \Delta P(b_0)}{\partial \nu_i} \kappa_{j}(b_0)\beta_{a,j}
+\Big(\sum_{l=1}^{2}\frac{\partial^2 \Delta P(b_0)}{\partial \tau_{l} \partial \tau_{j}} \beta_{a,l}\Big) =o(1),
\]
which together with the assumption $(\tilde{P})$ implies $\beta_{a,j}=o(1)$ for $j=1,2.$

\end{proof}

Finally, we prove Theorem \ref{th1.4}.
\begin{proof}[\textbf{Proof of Theorem \ref{th1.4}:}]
First, for large fixed $R$,  \eqref{abc8-29-1}, \eqref{a2---5}, \eqref{eq-d-710} and \eqref{eq-e-710} give
$$
C_{a}(x)- D_{a}(x)\ge \frac12,   \;\; |g_a(x)+E_{a}(x)|\le C e^{-\hat{\theta} R},\quad x\in \mathbb R^3 \backslash  B_{R\delta_a}(x^{(1)}_{a}),
$$
for some $\hat{\theta}>0.$

  Using the comparison principle, we get
\begin{equation*}
 \xi_a(x)= o(1),~\mbox{in}~ \R^3\backslash  \bigcup B_{R\delta_a}(x^{(1)}_{a}).
\end{equation*}

On the other hand, it follows from \eqref{8-27-2}, \eqref{8-28-44} and \eqref{add8-28-44} that
\[
\xi_{a}(x)=o(1),~\mbox{in}~ B_{R\delta_a}(x^{(1)}_{a}).
\]
This is in contradiction with $\|\xi_{a}\|_{L^{\infty}(\R^3)}=1.$
So $u^{(1)}_{a}(x)\equiv u^{(2)}_{a}(x)$ as $a$ goes to $+\infty.$
\end{proof}

\section{The non-existence of multi-peak solutions} \label{s6}
From \eqref{30-7-11}, we know that in order to prove Theorem \ref{th1.5}, it suffices to prove the following result.

\begin{Prop}\label{prop7-30-1}
Under the assumption \textup{($P$)},  there exists a small constant $\epsilon_{0}>0$ such that
  problem  \eqref{30-7-11}
has no $m$-peak solutions ($m\geq 2$) of the form
$$
\tilde{u}_{a}(x)=\sum_{i=1}^{m}U_{\epsilon,x_{a,i}}(x)+\varphi_{a}(x)
$$
with
$\|\varphi_{a}\|_{a}=O(\epsilon^{\frac{7}{2}})$, $x_{a,i}\rightarrow b_{i}$ as $0< \epsilon\leq \epsilon_{0}$
for each $i=1,...,m,$ and $b_i\neq b_j$ for $i\neq j$,
where $U_{\epsilon,x_{a,i}}(x):=(1+\epsilon^{2}P_{i})U\Big(\frac{\sqrt{1+\epsilon^2P_i}(x-x_{a, i})}\epsilon \Big).$
\end{Prop}

Let $\tilde{u}_a$ be a $m$-peak  solution of \eqref{30-7-11}.
Then for any small fixed $\rho>0$, from \eqref{7-30-1} we find
\begin{equation}\label{7-30-2}
\begin{split}
\epsilon^2&\int_{B_{\rho}(x_{a,i})} \frac{\partial P(x)}{\partial x_j}\big(\tilde{u}_a\big)^2  dx
\\=&\underbrace{-2\epsilon^2\int_{\partial B_{\rho}(x_{a,i})}\frac{\partial \tilde{u}_a}{\partial \bar \nu}\frac{\partial \tilde{u}_a}{\partial x_j} d\sigma}_{:=\widetilde{C}_{1}}
+\underbrace{\epsilon^2\int_{\partial B_{\rho}(x_{a,i})}|\nabla \tilde{u}_a|^2\bar\nu_j(x) d\sigma}_{:=\widetilde{C}_{2}}\\&
 +\underbrace{\int_{\partial B_{\rho}(x_{a,i})}
\big(1+\epsilon^2P(x)\big)\big(\tilde{u}_a\big)^2\bar\nu_j(x)d\sigma}_{:=\widetilde{C}_{3}}
-\underbrace{\frac{1}{8\pi\epsilon^{2}}\int_{\partial B_{\rho}(x_{a,i})}\int_{\R^{3}}
\frac{(\tilde{u}_a)^2}{|x-y|}dy(\tilde{u}_a)^2\bar\nu_j(x)d\sigma}_{:=\widetilde{C}_{4}}
\\\quad&
-\underbrace{\frac{1}{8\pi\epsilon^{2}}\int_{ B_{\rho}(x_{a,i})}\int_{\R^{3}}
\frac{(\tilde{u}_a(y))^2(\tilde{u}_a(x))^2(x_{j}-y_{j})}{|x-y|^{3}}dydx}_{:=\widetilde{C}_{5}},
\end{split}
\end{equation}
where $j=1,2,3$ and $\bar \nu(x)=\big(\bar\nu_{1}(x),\bar\nu_2(x),\bar\nu_3(x)\big)$ is the outward unit normal of $\partial B_{\rho}(x_{a,i}).$

Let  $\tilde \varphi_{a}(x) =\varphi_{a}(\epsilon x + x_{a,i})$. Then, $\tilde \varphi_{a}$ satisfies
$\|\tilde \varphi_{a}\|_a=O\big(\epsilon^2\big).$ Using the Moser iteration,
we can prove
$\|\tilde \varphi_{a}\|_{L^\infty(\mathbb R^3)}=o(1).$
 From this fact and the comparison theorem, similar to Proposition 2.2 in \cite{LPW-20CV},
  we can prove the following estimates for $\tilde u_a(x)$ away from the concentrated points $b_1,\cdots,b_m.$

\begin{Prop}
Suppose that $\tilde u_a(x)$ is a $m$-peak solution  of  \eqref{30-7-11} concentrated at $b_1,\cdots,b_m.$ Then
for any fixed $R\gg 1,$ there exist some $\theta>0$ and $C>0,$ such that
\begin{equation}\label{a7-31-1}
|\tilde u_a(x)|+|\nabla \tilde u_a(x)|\leq C\sum^m_{i=1}e^{-\theta |x-x_{a,i}|/\epsilon},~\mbox{for}~
x\in \R^3\backslash \bigcup^m_{i=1}B_{R \epsilon}(x_{a,i}).
\end{equation}
\end{Prop}
From the proof of Proposition 4.2 in \cite{LPW-20CV}, we have the following result.

\begin{Lem}\label{aprop-7-30-2}
For the small fixed constant $\bar \rho>0$ and any $\rho\in (\bar \rho, 2\bar \rho)$,
then there exist $i_0\in \{1,\cdots,m\},$ $j_0\in \{1,2,3\}$ and $C^{*}=C(i_0,j_{0})\neq 0$ such that
\begin{flalign}\label{claim1}
\widetilde{C}_{5}=C^{*}\epsilon^{4}+o(\epsilon^4).
\end{flalign}
\end{Lem}

Now we are ready to prove Proposition \ref{prop7-30-1}.

\begin{proof}[\textbf{Proof of Proposition \ref{prop7-30-1}}]
Here we prove it by  contradiction.
Assume that \eqref{30-7-11} has a $m$-peak solution $\tilde{u}_{a}(x)$.
Then taking $i=i_0$ and $j=j_0$ as in Lemma \ref{aprop-7-30-2}, it follows from \eqref{7-30-2} and \eqref{claim1} that
\begin{equation*}\label{7-30-3}
\begin{split}
\epsilon^2\int_{B_{\rho}(x_{a,i})} \frac{\partial P(x)}{\partial x_j}\big(\tilde{u}_a\big)^2  dx
=C^{*}\epsilon^{4}+o(\epsilon^{4})+O(e^{-\frac{\theta}{\epsilon}}),
\end{split}
\end{equation*}
since similar to \eqref{eqd-7-10-1-1} and by \eqref{a7-31-1} we can prove
\begin{equation*}\label{7-30-3}
\begin{split}
\sum_{i=1}^{4}\widetilde{C}_{i}\leq C\int_{\partial B_{\rho}(x_{a,i})}
 \big(\epsilon^2|\nabla \tilde{u}_a|^2+\tilde{u}^{2}_a(x)\big)d\sigma
 \leq C\int_{\partial B_{\rho}(x_{a,i})}\sum_{j=1}^{m}e^{\frac{-\theta \rho}{\epsilon}} d\sigma
 =O(e^{-\frac{\theta \rho}{\epsilon}}),
\end{split}
\end{equation*}
where we use the fact that
$\big\{x:\partial B_{\rho}(x_{a,i})\big\}\subset \big\{x:\R^{3}\setminus \bigcup_{j=1}^{m} B_{\frac{\rho}{2}}(x_{a,j})\big\}.$
Since $x_{a,i}\to b_i\in \Gamma_{i},$ we find that there is a $t_a\in [P_{i}, P_{i}+\sigma]$ if $\Gamma$
is a local minimum set of $P(x)$, or $t_a\in [P_{i}-\sigma, P_{i}]$ if $\Gamma_{i}$
is a local maximum set of $P(x)$, such that $x_{a,i}\in \Gamma_{t_a,i}.$
Let $\tau_{a,i}$ be the unit tangential vector of $\Gamma_{t_a,i}$ at $x_{a,i}.$  Then
$$
G(x_{a,i}) =0,~\mbox{where}~G(x)=  \bigl\langle \nabla P(x), \tau_{a,i}\bigr\rangle.
$$
We have the following expansion:
\begin{equation*}
\begin{split}
 G(x)=& \langle\nabla G( x_{a,i}), x-x_{a,i}\rangle+
\frac{1}{2}\big\langle \langle \nabla^2 G( x_{a,i}), x- x_{a,i}\rangle,x- x_{a,i}\big\rangle
+o\big(|x- x_{a,i}|^2\big),~\mbox{for}~x\in B_{\rho}(x_{a,i}).
\end{split}
\end{equation*}
Then it follows from  \eqref{30-31-7}, the above expansion and $\|\varphi_{a}\|_{a}=O(\epsilon^{\frac{7}{2}})$ that
\begin{equation}\label{luo-6-7-30}
\begin{split}
\int_{\mathbb R^3} G(x)U^{2}_{\epsilon,x_{a,i}}(x)=
&
\int_{B_{\rho}(x_{a,i})} G(x)U^{2}_{\epsilon,x_{a,i}}(x)+
O\big(e^{-\frac{\tilde{\theta}}{\epsilon}}\big)
\\=&
-2\int_{B_{\rho}(x_{a,i})} G(x)U_{\epsilon,x_{a,i}}(x) \varphi_{a}
- \int_{B_{\rho}(x_{a,i})}G(x) \varphi^2_{a}+C^{*}\epsilon^{2}+o(\epsilon^{2})\\=&
O\Big( \big[\epsilon^{\frac{5}{2}} |\nabla G(x_{ a,i})| +\epsilon^{\frac{7}{2}}\big]\|\varphi_{a}\|_a+\epsilon|\nabla G(x_{ a,i})|\cdot
\|\varphi_{a}\|^2_a\Big)+C^{*}\epsilon^{2}+o(\epsilon^{2})\\
=&C^{*}\epsilon^{2}+
o\big( \epsilon^{2}\big),
\end{split}
\end{equation}
where $0<\tilde{\theta}<\theta$ small.

On the other hand, noting that $ G(x_{a,i})=0,$ it is easy to check
\begin{equation}\label{06-09-1-7-30}
 \int_{\mathbb R^3} G(x)U^{2}_{\epsilon,x_{a,i}}(x)=
 \frac{1}{6} \epsilon^{5}\Delta G(x_{a,i}) \displaystyle\int_{\R^3}|x|^2U^2
+O\bigl (\epsilon^{7}\bigr)=O\bigl (\epsilon^{5}\bigr).
\end{equation}
Then it follows from \eqref{luo-6-7-30} and \eqref{06-09-1-7-30} that
$$C^{*}=o(1),$$
which contradicts with $C^{*}\neq 0.$
Therefore \eqref{30-7-11} has no $m$-peak solution $\tilde{u}_{a}(x).$
\end{proof}

With Proposition \ref{prop7-30-1} at hand, we can prove Theorem \ref{th1.5} at once.

\begin{proof}[
\textbf{Proof of Theorem \ref{th1.5}}]
Theorem \ref{th1.5} follows directly from Proposition \ref{prop7-30-1}
and the relation between the solutions of equation \eqref{8-18-1} and  equation \eqref{30-7-11}.
\end{proof}

\section*{Appendix}

\appendix

\section{Some known results}\label{sa}

\renewcommand{\theequation}{A.\arabic{equation}}
In this section, we give some known results which are used before.
Denote
$$
H_{\epsilon}=\Big\{u\in H^{1}(\R^{3}),\int_{\R^{3}}(\epsilon^{2}|\nabla u|^{2}+P(x)u^{2}(x))dx<\infty\Big\}
$$
and the corresponding norm
$$
\|u\|_{\epsilon}=(u(x),u(x))_{\epsilon}^{\frac{1}{2}}
=\Big(\int_{\R^{3}}(\epsilon^{2}|\nabla u|^{2}+P(x)u^{2}(x))dx\Big)^{\frac{1}{2}}.
$$

\begin{Lem}\label{lem-add1-1}(Lemma 2.1, \cite{GW-20ANS})
For each $u\in L^{q}(\R^{3})(2\leq q\leq 6),$ we have
$$
|u|_{q}\leq C\delta_{a}^{3(\frac{1}{q}-\frac{1}{2})}\|u\|_{\delta_{a}}.
$$
\end{Lem}

\begin{Lem}\label{lem-7-9-1}(Lemma 2.2, \cite{GW-20ANS}, Lamma A.5, \cite{LPW-20CV})
For any $u_{i}(i=1,2,3,4)\in H_{\epsilon},$ then
$$
\int_{\R^{3}}\int_{\R^{3}}\frac{u_{1}(x)u_{2}(x)u_{3}(y)u_{4}(y)}{|x-y|}dx\,dy
\leq C\epsilon^{-1}
\|u_{1}\|_{\epsilon}\|u_{2}\|_{\epsilon}\|u_{3}\|_{\epsilon}\|u_{4}\|_{\epsilon}.
$$
\end{Lem}

\begin{Lem}(Lemma A.6, \cite{LPW-20CV})\label{lem-add-7-11}
For any $u_1,u_2,u_3,u_4\in H^1(\R^3),$  then
\begin{flalign*}\label{aa9}
\int_{\R^3}\int_{\R^3}\frac{u_1(x)u_2(x)u_3(y)u_4(y)}{|x-y|}dx dy
\leq C\|u_1\|_{H^1(\R^3)}\|u_2\|_{H^1(\R^3)}\|u_3\|_{H^1(\R^3)}\|u_4\|_{H^1(\R^3)}.
\end{flalign*}
\end{Lem}

\begin{Lem}(Lemma B.2, \cite{LPW-20CV})
For any fixed $R>0,$ it holds
\begin{equation}\label{eq-d-710}
D_{a}(x)=o(1)R+O\Big(\frac{1}{R}\Big),\,\,\,\text{for}\,\,x\in \R^{3}\setminus B_{R\delta_{a}(x_{a})},
\end{equation}
and
\begin{equation}\label{eq-e-710}
E_{a}(x)=O(e^{-\tilde{\theta}R}),\,\,\,\text{for}\,\,x\in \R^{3}\setminus B_{R\delta_{a}(x_{a})},\,\,\text{and~some}\,\,\,\tilde{\theta}>0.
\end{equation}
\end{Lem}

Now let $\widetilde{\Gamma}\in C^2$ be a closed hypersurface in $\mathbb R^3.$ For $y\in \widetilde{\Gamma},$ let $\nu(y)$ and $T(y)$ denote respectively the outward unit  normal to $\widetilde{\Gamma}$ at $y$ and the tangent hyperplane to $\widetilde{\Gamma}$ at $y.$ The curvatures of $\widetilde{\Gamma}$ at a fixed point $y_0\in\widetilde{\Gamma}$ are determined as follows. By a rotation of coordinates, we can assume that
$y_0=0$ and $\nu(0)$ is the $x_3$-direction, and $x_j$-direction is the $j$-th principal direction.

In some neighborhood $\mathcal{N}=\mathcal{N}(0)$ of $0,$  we have
\[
\widetilde{\Gamma}=\bigl\{  x:   x_3=\phi(x')\bigr\},
\]
where $x'=(x_1,x_{2}),$
 \[
 \phi(x') =\frac12 \sum_{j=1}^{2} \kappa_j x_j^2 +  O(|x'|^3),
 \]
 where
 $\kappa_j,$ is  the $j$-th principal curvature of  $\widetilde{\Gamma}$ at $0.$
 The Hessian matrix $[D^2 \phi(0)]$  is given by
\begin{equation*}
[D^2 \phi(0)]=diag [\kappa_1,\kappa_{2}].
\end{equation*}

Suppose that  $W$ is a smooth function, such  that  $W(x)=constant$  for all $x\in \widetilde{\Gamma}.$
It follows from \cite{LPY-19} that

\begin{Lem}(Lemma B.1, \cite{LPY-19})
We have
 \begin{equation*}\label{7-19-29}
 \frac{\partial W(x)}{\partial x_l}\Bigr|_{x=0}=0, ~\mbox  ~l=1,2,
 \end{equation*}
  \begin{equation}\label{ab7-19-29}
  \frac{\partial^2 W(x)}{\partial x_m\partial x_l }\Bigr|_{x=0}=-\frac{\partial W\big(x\big)}{\partial x_3}\Bigr|_{x=0}\kappa_m \delta_{ml}, ~\mbox{for}~m, l=1,2,
 \end{equation}
 where  $\kappa_1,\kappa_{2},$ are  the principal curvatures of  $\widetilde{\Gamma}$ at $0.$
\end{Lem}

\section{Linearization and A Pohozaev identity}

\renewcommand{\theequation}{B.\arabic{equation}}
\begin{Lem}\label{lem-bbb}
Let $\xi_0$ be the solution of following system:
\begin{equation}\label{06-10-23}
\begin{split}
&-\Delta \xi(x)+\xi(x)
-\frac{1}{8\pi}\int_{\R^{3}}\frac{U^{2}(y)}{|x-y|}dy\xi(x)
-\frac{1}{4\pi}\int_{\R^{3}}\frac{U(y)\xi(y)}{|x-y|}dy\,U(x)\\
&=-\frac{1}{4\pi a_*}U(x)\int_{\R^3}\int_{\R^3}\frac{U^2(x)U(y) \xi(y)}{|x-y|}dx\,dy.
\end{split}
\end{equation}
Then it holds
\begin{equation}\label{06-10-25}
\xi(x)=  \sum^3_{j=0}\gamma_{j}\psi_j,
\end{equation}
where  $\gamma_{j}$ are some constants,
\begin{equation}\label{aaaaa}
\psi_0=2U+x\cdot\nabla U,~\psi_j=\frac{\partial U}{\partial x_j},~\mbox{for}~j=1,2,3.
\end{equation}
\end{Lem}
\begin{proof}
We set $\bar L(u):=-\Delta u(x)+u(x)
-\frac{1}{8\pi}\int_{\R^{3}}\frac{U^{2}(y)}{|x-y|}dy\,u(x)
-\frac{1}{4\pi}\int_{\R^{3}}\frac{U(y)u(y)}{|x-y|}dy\,U(x).$
It is obvious that $\bar  L (\psi_j)=0$ and $\psi_j$
 is the solution of \eqref{06-10-23}, for $j=1,2,3.$
Also,
then using \eqref{06-13-1}, we find that $\psi_{0}$ is also the solution of \eqref{06-10-23}. And we know that $\psi_{0},\psi_{1},\psi_{2},\psi_{3}$ are linearly independent. Then we get \eqref{06-10-25}.
\end{proof}

\begin{Prop}
Let $\tilde{u}_{a}(x)$ be the solution of equation \eqref{30-7-11}. Then we have following local Pohozaev identity:
\begin{flalign}\label{7-30-1}
\begin{split}
\epsilon^{2}\int_{\Omega}\frac{\partial P(x)}{\partial x_j}\tilde{u}_{a}^2(x)dx=&
-2\epsilon^2\int_{\partial\Omega}\frac{\partial \tilde{u}_{a}(x)}{\partial \nu}\frac{\partial \tilde{u}_{a}(x)}{\partial x_j}\mathrm{d}\sigma
+\epsilon^2\int_{\partial\Omega}|\nabla\tilde{u}_{a}(x)|^{2}\nu_{j}(x)\mathrm{d}\sigma
\\
&+\int_{\partial \Omega}(1+\epsilon^{2}P(x))\tilde{u}_{a}^2(x)\nu_j(x)d
\sigma
\\&-\frac{1}{8\pi \epsilon^2}\int_{\partial \Omega} \int_{\R^3}\frac{\tilde{u}_{a}^2(y)\tilde{u}_{a}^2(x)}{|x-y|}\nu_j(x)dy d
\sigma-
\frac{1}{8\pi \epsilon^2}\int_{\Omega}\int_{\R^3}\tilde{u}_{a}^2(y)\tilde{u}_{a}^2(x)\frac{x_j-y_j}{|x-y|^3}dy dx,
\end{split}
\end{flalign}
where  $\Omega$ is a bounded open domain of $\R^3$, $j=1,2,3,$ $\nu(x)=\big(\nu_{1}(x),\nu_2(x),\nu_3(x)\big)$ is the outward unit normal of $\partial \Omega$ and $x_j, y_j$ are the $j$-th components of $x, y.$
 \end{Prop}
 \begin{proof}
Since \eqref{7-30-1} is obtained just by multiplying $\frac{\partial \tilde{u}_{a}(x)}{\partial x_j}$ on both sides of \eqref{1.2} and integrating on $\Omega,$
here we omit the details.
 \end{proof}

\section{Some basic and useful estimates}\label{sd}

\renewcommand{\theequation}{C.\arabic{equation}}

In this section, we mainly give some basic and useful estimates which are used before.
\begin{Lem}
There holds
 \begin{equation}\label{eqd-7-10-1}
 \begin{split}
 C_{1}+C_{2}+C_{3}+C_{4}+C_{5}=O\big(e^{-\frac{\theta}{\epsilon}}\big).
\end{split}
\end{equation}
\end{Lem}

\begin{proof}
From \eqref{eq-d-710} and \eqref{2--5}, we have
\begin{equation}\label{eqd-7-10-1-1}
 \begin{split}
 C_{1}+C_{2}+C_{3}+C_{4}\leq C\int_{\partial B_{\rho}(x_{a})}
 (\epsilon^2|\nabla \tilde{u}_a|^2+\tilde{u}^{2}_a(x))d\sigma
 =O\big(e^{-\frac{\theta}{\epsilon}}\big).
\end{split}
\end{equation}
By symmetry and \eqref{2--5}, we have
\begin{equation}\label{eqd-7-10-1-2}
 \begin{split}
C_{5}&=\frac{1}{8\pi\epsilon^{2}}\int_{ \R^{3}}\int_{\R^{3}}
\frac{(\tilde{u}_a(y))^2(\tilde{u}_a(x))^2(x_{j}-y_{j})}{|x-y|^{3}}dydx\\
&
\quad-\frac{1}{8\pi\epsilon^{2}}\int_{ B^{C}_{\rho}(x_{a})}\int_{\R^{3}}
\frac{(\tilde{u}_a(y))^2(\tilde{u}_a(x))^2(x_{j}-y_{j})}{|x-y|^{3}}dydx\\
& =O\big(e^{-\frac{\theta}{\epsilon}}\big).
\end{split}
\end{equation}
By
\eqref{eqd-7-10-1-1} and \eqref{eqd-7-10-1-2}, \eqref{eqd-7-10-1} is true.
\end{proof}

\begin{Lem}
For  any small and fixed $\rho>0,$
if $x\in B_{\rho\delta^{-1}_{a}}\sqrt{\textbf{p}_{0}}(0),$ then we have
\begin{equation}\label{f}
\begin{split}
F_{1}+F_{2}
=O(\delta^{4}_{a}).
\end{split}
\end{equation}
\end{Lem}

\begin{proof}
From \eqref{8-27-26} and Lemma \ref{lem-add1-1}, by H\"older inequality we have
\begin{equation}\label{eq-f2}
\begin{split}
F_{1}&=\frac{1}{4\pi \textbf{p}_{0} }
\int_{\R^{3}}
\frac{U(y)\tilde{\varphi}^{(1)}_{a}(\frac{\delta_a y}{\sqrt{\textbf{p}_{0}}}+x^{(1)}_{a}) }{|x-y|}dy
=\frac{1}{4\pi \textbf{p}_{0}\delta^{2}_{a} }
\int_{\R^{3}}
\frac{
U(\frac{\sqrt{\textbf{p}_{0}}}{\delta_{a}}(y-x^{(1)}_{a})\tilde{\varphi}^{(1)}_{a}(y) }{|y-x^{(1)}_{a}-\frac{\delta_{a}}{\sqrt{\textbf{p}_{0}}}x|}dy\\
&\leq \frac{1}{4\pi \textbf{p}_{0}\delta^{2}_{a} }
\Big|
U(\frac{\sqrt{\textbf{p}_{0}}}{\delta_{a}}(y-x^{(1)}_{a}))\Big|_{3}
|\tilde{\varphi}^{(1)}_{a}(y) |_{6}
\Big(\int_{|y-x^{(1)}_{a}-\frac{\delta_{a}}{\sqrt{ \textbf{p}_{0}}}x|\leq c_{0}\delta_{a}}
\frac{1}{|y-x^{(1)}_{a}-\frac{\delta_{a}}{\sqrt{ \textbf{p}_{0}}}x|^{2}}\Big)^{\frac{1}{2}}\\
&\quad
+\frac{1}{4\pi \textbf{p}_{0}\delta^{3}_{a} c_{0}}
\Big\|U(\frac{\sqrt{\textbf{p}_{0}}}{\delta_{a}}(y-x^{(1)}_{a}))\Big\|_{\delta_{a}}
\|\tilde{\varphi}^{(1)}_{a}(y)\|_{\delta_{a}}\\
&\leq  C\delta^{-2}_{a}\delta_{a}
(\delta^{-1}_{a}\|\tilde{\varphi}^{(1)}_{a}(y)\|_{\delta_{a}})\delta^{\frac{1}{2}}_{a}+
C\delta_{a}^{-3}\delta^{\frac{3}{2}}_{a}\delta^{\frac{11}{2}}_{a}
=O(\delta^{4}_{a}),
\end{split}
\end{equation}
and
\begin{equation}\label{eq-f3}
\begin{split}
F_{2}&=\frac{1}{8\pi \textbf{p}^{2}_{0} }
\int_{\R^{3}}\frac{\big(\tilde{\varphi}_a^{(1)}(\frac{\delta_a }{\sqrt{\textbf{p}_{0}}}y+x^{(1)}_{a})\big)^{2} }{|x-y|}dy
=\frac{1}{8\pi \textbf{p}_{0}\delta^{2}_{a} }
\int_{\R^{3}}
\frac{
(\tilde{\varphi}^{(1)}_{a}(y) )^{2}}{|y-x^{(1)}_{a}-\frac{\delta_{a}}{\sqrt{\textbf{p}_{0}}}x|}dy\\
&\leq \frac{1}{8\pi \textbf{p}_{0}\delta^{2}_{a} }
|\tilde{\varphi}^{(1)}_{a}(y) |^{2}_{4}
\Big(\int_{|y-x^{(1)}_{a}-\frac{\delta_{a}}{\sqrt{ \textbf{p}_{0}}}x|\leq c_{0}\delta_{a}}
\frac{1}{|y-x^{(1)}_{a}-\frac{\delta_{a}}{\sqrt{ \textbf{p}_{0}}}x|^{2}}\Big)^{\frac{1}{2}}\\
&\quad
+\frac{1}{4\pi \textbf{p}_{0}\delta^{3}_{a} c_{0}}
\|\tilde{\varphi}^{(1)}_{a}(y)\|^{2}_{\delta_{a}}\\
&\leq  C\delta^{-2}_{a}(\delta_{a}^{-\frac{3}{4}}\|\tilde{\varphi}^{(1)}_{a}(y)\|_{\delta_{a}})^{2}\delta^{\frac{1}{2}}_{a}+
C\delta_{a}^{-3}\delta^{11}_{a}=O(\delta^{8}_{a}).
\end{split}
\end{equation}
Hence, it follows from \eqref{eq-f2} to \eqref{eq-f3} that \eqref{f} holds.
\end{proof}

\begin{Lem}
For  any small and fixed $\rho>0,$
if $x\in B_{\rho\delta^{-1}_{a}}\sqrt{\textbf{p}_{0}}(0),$
there holds
\begin{equation}\label{g}
\begin{split}
G_{1}+G_{2}
=O\Big(\delta_a+\tilde{\varphi}^{(2)}_{a}(\frac{\delta_a}{\sqrt{\textbf{p}_{0}}} x+x^{(1)}_{a})\Big).
\end{split}
\end{equation}

\end{Lem}

\begin{proof}
By the mean value theorem and \eqref{8-27-38}, we have
\begin{equation}\label{eq-g1}
\begin{split}
G_{1}&=\frac{1}{4\pi \textbf{p}_{0}}\big(u_a^{(2)}(\frac{\delta_a x}{\sqrt{\textbf{p}_{0}}}+x^{(1)}_{a,i})
-\textbf{p}_{0}U(x)\big)\int_{\R^{3}}\frac{U(y)\bar{\xi}_{a}(y)}{|x-y|}dy
\\
&=\frac{1}{4\pi \textbf{p}_{0}}\Big(\nabla U\big(\zeta x+(1-\zeta)\frac{\sqrt{\textbf{p}_{0}}}{\delta_a}(x^{(1)}_{a}-x^{(2)}_{a})\big)\cdot \frac{\sqrt{\textbf{p}_{0}}}{\delta_a}(x^{(1)}_{a}-x^{(2)}_{a}\big)
\\
&\quad
+\tilde{\varphi}^{(2)}_{a}(\frac{\delta_a x}{\sqrt{\textbf{p}_{0}}} +x^{(1)}_{a})
\Big)\int_{\R^{3}}\frac{U(y)\bar{\xi}_{a}(y)}{|x-y|}dy
\\
&\leq
C\Big(\delta_{a}
+\tilde{\varphi}^{(2)}_{a}(\frac{\delta_a}{\sqrt{\textbf{p}_{0}}} x+x^{(1)}_{a})\Big)
\int_{\R^{3}}
\frac{U(y)\bar{\xi}_{a}(y)}{|x-y|}dy\\
&=O\Big(\delta_{a}+\tilde{\varphi}^{(2)}_{a}(\frac{\delta_a}{\sqrt{\textbf{p}_{0}}} x+x^{(1)}_{a})\Big),
\,\,\,\text{for}\,\,\zeta \in(0,1),
\end{split}
\end{equation}
since
\begin{equation}\label{kxi}
\begin{split}
\int_{\R^{3}}
\frac{U(y)\bar{\xi}_{a}(y)}{|x-y|}dy
&\leq  \int_{\R^{3}}
\frac{U(y)}{|x-y|}dy\\
&\leq |U(y)|_{2}\Big(\int_{|x-y|\leq 2R}\frac{1}{|x-y|^{2}}dy\Big)^{\frac{1}{2}}
+\frac{1}{2R}\int_{\R^{3}}U(y)dy\\
&=O\Big(R^{\frac{1}{2}}+\frac{1}{R}\Big)
=O(1).
\end{split}
\end{equation}

We can also check
\begin{equation}\label{eq-g2}
\begin{split}
 G_{2}
 &=
 \frac{u_a^{(2)}(\frac{\delta_a x}{\sqrt{\textbf{p}_{0}}}+x^{(1)}_{a})}{8\pi \textbf{p}^{2}_{0}}\int_{\R^{3}}
\frac{\big((u_a^{(1)}+u_a^{(2)})(\frac{\delta_a y}{\sqrt{\textbf{p}_{0}}}+x^{(1)}_{a})-2\textbf{p}_{0}U(y)\big)\bar{\xi}_{a}(y)}{|x-y|}dy
 \\
 &\leq C\frac{|x^{(1)}_{a,i}-x^{(2)}_{a}|}{\delta_a}
 \int_{\R^{3}}\frac{|\nabla U\big(\zeta y+(1-\zeta)\frac{\sqrt{\textbf{p}_{0}}}{\delta_a}(x^{(1)}_{a}-x^{(2)}_{a})\big)|\cdot
 |\bar{\xi}_{a}(y)|}{|x-y|}dy
 \\
 &\quad
 +C\sum_{l=1}^{2}\int_{\R^{3}}\frac{|\tilde{\varphi}^{(l)}_{a}(\frac{\delta_a y}{\sqrt{\textbf{p}_{0}}}+x^{(1)}_{a,i})|\cdot
 |\bar{\xi}_{a}(y)|}{|x-y|}dy\\
 &=O(\delta_{a}),
\end{split}
\end{equation}
since similar to \eqref{kxi}, we have
\begin{equation*}\label{eq-7-4-1}
\begin{split}
\int_{\R^{3}}
\frac{|\nabla U\big(\zeta y+(1-\zeta)\frac{\sqrt{\textbf{p}_{0}}}{\delta_a}(x^{(1)}_{a}-x^{(2)}_{a})\big)|\cdot
|\bar{\xi}_{a}(y)|}{|x-y|}dy=O(1),
\end{split}
\end{equation*}
and
\begin{equation*}\label{eq-7-4-3}
\begin{split}
&\int_{\R^{3}}
\frac{\tilde{\varphi}^{(l)}_{a}(\frac{\delta_a y}{\sqrt{\textbf{p}_{0}}}+x^{(1)}_{a})\bar{\xi}_{a}(y)}{|x-y|}dy\\
&\leq \int_{|y|\leq 2\frac{\sqrt{\textbf{p}_{0}}}{\delta_{a}}\rho}\frac{\tilde{\varphi}^{(l)}_{a}(\frac{\delta_a y}{\sqrt{\textbf{p}_{0}}}+x^{(1)}_{a})\bar{\xi}_{a}(y)}{|x-y|}dy
+\int_{|y|\geq 2\frac{\sqrt{\textbf{p}_{0}}}{\delta_{a}}\rho}\frac{\tilde{\varphi}^{(l)}_{a}(\frac{\delta_a y}{\sqrt{\textbf{p}_{0}}}+x^{(1)}_{a})\bar{\xi}_{a}(y)}{|x-y|}dy\\
&\leq C |\tilde{\varphi}^{(l)}_{a}(\frac{\delta_a y}{\sqrt{\textbf{p}_{0}}}+x^{(1)}_{a})|_{3}|\bar{\xi}_{a}(y)|_{6}\Big(\int_{|x-y|\leq 3\frac{\sqrt{\textbf{p}_{0}}}{\delta_{a}}\rho}
\frac{1}{|x-y|^{2}}dy\Big)^{\frac{1}{2}}
+C\delta_{a}|\tilde{\varphi}^{(l)}_{a}(\frac{\delta_a y}{\sqrt{\textbf{p}_{0}}}+x^{(1)}_{a})|_{2}|\bar{\xi}_{a}(y)|_{2}
\\
&\leq C\delta^{-1}_{a}|\tilde{\varphi}^{(l)}_{a}(y)|_{3}\delta^{-\frac{1}{2}}_{a}
+C\delta_{a}\delta^{-\frac{3}{2}}_{a}|\tilde{\varphi}^{(l)}_{a}(y)|_{2}
\\
&\leq C\delta^{-\frac{3}{2}}_{a}(\delta^{-\frac{1}{2}}_{a}\delta^{\frac{11}{2}}_{a})+C\delta_{a}^{-\frac{1}{2}}\delta^{\frac{11}{2}}_{a}
=O(\delta_{a}^{\frac{7}{2}}).
\end{split}
\end{equation*}
Then \eqref{g} follows from \eqref{eq-g1} and \eqref{eq-g2}.
\end{proof}

\begin{Lem}
There holds
\begin{equation}\label{uu}
\begin{split}
&\int_{\R^{3}}\int_{\R^{3}}\frac{\big[(u_a^{(1)})^{2}(x)u_a^{(2)}(y)
+(u_a^{(2)})^{2}(x)u_a^{(1)}(y)\big]\xi_{a}(y)}{|x-y|}dxdy
\\
&=2\delta^{5}_a\textbf{p}^{\frac{1}{2}}_{0}
\int_{\R^{3}}\int_{\R^{3}}\frac{U^{2}(x)U(y)\bar{\xi}_{a}(y)}{|x-y|}dxdy
+O(\delta_a^{6}).
\end{split}
\end{equation}
\end{Lem}

\begin{proof}
By direct computations, from \eqref{06-13-1} we have
\begin{equation}\label{uu-1-1}
\begin{split}
\int_{\R^{3}}&\int_{\R^{3}}\frac{(u_a^{(1)})^{2}(x)u_a^{(2)}(y)
\xi_{a}(y)}{|x-y|}dxdy\\
=&\delta_{a}^{5}\textbf{p}^{\frac{1}{2}}_{0}\int_{\R^{3}}\int_{\R^{3}}
\frac{U^{2}(x)U(y)\bar{\xi}_{a}(y)}{|x-y|}dxdy
\\
&
+\frac{\delta_{a}^{5}}{\textbf{p}^{\frac{1}{2}}_{0}}
\underbrace{\int_{\R^{3}}\int_{\R^{3}}
\frac{U^{2}(x)(u^{(2)}_{a}
(\frac{\delta_{a}}{\sqrt{\textbf{p}_{0}}}y+x^{(1)}_{a})-\textbf{p}_{0}U(y))\bar{\xi}_{a}(y)}{|x-y|}
dxdy}_{:=H_{1}} \\&
+
\delta_{a}^{5}\textbf{p}^{\frac{1}{2}}_{0}\underbrace{\int_{\R^{3}}\int_{\R^{3}}
\frac{\tilde{\varphi}^{(1)}_{a}(\frac{\delta_{a}}{\sqrt{\textbf{p}_{0}}}y+x^{(1)}_{a})
\big(u^{(2)}_{a}
(\frac{\delta_{a}}{\sqrt{\textbf{p}_{0}}}y+x^{(1)}_{a})+\textbf{p}_{0}U(x)\big)u^{(2)}_{a}
(\frac{\delta_{a}}{\sqrt{\textbf{p}_{0}}}y+x^{(1)}_{a})\bar{\xi}_{a}(y)}{|x-y|}dxdy}_{:=H_{2}}
\\=
& \delta_{a}^{5}\textbf{p}^{\frac{1}{2}}_{0}\int_{\R^{3}}\int_{\R^{3}}
\frac{U^{2}(x)U(y)\bar{\xi}_{a}(y)}{|x-y|}dxdy
+O(\delta^{6}_{a}),
\end{split}
\end{equation}
since
by Lemmas \ref{lem-7-9-1}, \ref{lem-add-7-11} and \eqref{8-27-26}, we have
\begin{equation*}\label{eq-h1}
\begin{split}
H_{1}&=
\int_{\R^{3}}\int_{\R^{3}}
\frac{U^{2}(x)\textbf{p}_{0}
\big(U(y+\frac{\sqrt{\textbf{p}_{0}}}{\delta_{a}}(x^{(1)}_{a}-x^{(2)}_{a}))-U(y)\big)\bar{\xi}_{a}(y)}{|x-y|}dxdy
\\
&\quad+\int_{\R^{3}}\int_{\R^{3}}
\frac{U^{2}(x)\tilde{\varphi}_{a}^{(2)}
\big(\frac{\delta_{a}}{\sqrt{\textbf{p}_{0}}}y+x^{(1)}_{a}\big)\bar{\xi}_{a}(y)}{|x-y|}dxdy
\\
&=O(\delta_{a})\int_{\R^{3}}\int_{\R^{3}}
\frac{U^{2}(x)
|\nabla U(\zeta y+(1-\zeta)\frac{\sqrt{\textbf{p}_{0}}}{\delta_{a}}(x^{(1)}_{a}-x^{(2)}_{a}))||\bar{\xi}_{a}(y)|}{|x-y|}dxdy
\\
&\quad+C\delta_{a}^{-1}\|U\|_{\delta_{a}}^{2}
\Big\|\tilde{\varphi}_{a}^{(2)}
\big(\frac{\delta_{a}}{\sqrt{\textbf{p}_{0}}}y+x^{(1)}_{a}\big)\Big\|_{\delta_{a}}\|\bar{\xi}_{a}\|_{\delta_{a}}
\\
&=O(\delta_{a})\|U\|^{2}
\Big\|\nabla U(\zeta y+(1-\zeta)\frac{\sqrt{\textbf{p}_{0}}}{\delta_{a}}(x^{(1)}_{a}-x^{(2)}_{a}))\Big\|\|\bar{\xi}_{a}\|
+O(\delta^{3}_{a})\\
&=O(\delta_{a}),\,\,\,\text{~for~ some}\,\,\zeta\in(0,1),
\end{split}
\end{equation*}
and
\begin{equation*}\label{eq-h2}
\begin{split}
H_{2}&=2\textbf{p}^{2}_{0}\int_{\R^{3}}\int_{\R^{3}}
\frac{U(x)\tilde{\varphi}^{(1)}_{a}(\frac{\delta_{a}}{\sqrt{\textbf{p}_{0}}}x+x^{(1)}_{a})
U(y+\frac{\sqrt{\textbf{p}_{0}}}{\delta_{a}}(x^{(1)}_{a}-x^{(2)}_{a}))\bar{\xi}_{a}(y)}{|x-y|}dxdy\\
&\quad+
2\textbf{p}_{0}\int_{\R^{3}}\int_{\R^{3}}
\frac{U(x)\tilde{\varphi}^{(1)}_{a}(\frac{\delta_{a}}{\sqrt{\textbf{p}_{0}}}x+x^{(1)}_{a})
\tilde{\varphi}^{(2)}_{a}(\frac{\delta_{a}}{\sqrt{\textbf{p}_{0}}}y+x^{(1)}_{a})\bar{\xi}_{a}(y)}{|x-y|}dxdy\\
&
\quad+\textbf{p}_{0}
\int_{\R^{3}}\int_{\R^{3}}
\frac{(\tilde{\varphi}^{(1)}_{a}(\frac{\delta_{a}}{\sqrt{\textbf{p}_{0}}}x+x^{(1)}_{a}))^{2}
U(x+\frac{\sqrt{\textbf{p}_{0}}}{\delta_{a}}(x^{(1)}_{a}-x^{(2)}_{a}))\bar{\xi}_{a}(y)}{|x-y|}dxdy
\end{split}
\end{equation*}

\begin{equation*}\label{eq-h2}
\begin{split}
&\quad
+\int_{\R^{3}}\int_{\R^{3}}
\frac{(\tilde{\varphi}^{(1)}_{a}(\frac{\delta_{a}}{\sqrt{\textbf{p}_{0}}}x+x^{(1)}_{a}))^{2}
\tilde{\varphi}^{(2)}_{a}(\frac{\delta_{a}}{\sqrt{\textbf{p}_{0}}}y+x^{(1)}_{a})\bar{\xi}_{a}(y)}{|x-y|}dxdy\\
&\leq C\delta^{-1}_{a}\|U\|^{2}_{\delta_{a}}
\Big\|\tilde{\varphi}^{(1)}_{a}(\frac{\delta_{a}}{\sqrt{\textbf{p}_{0}}}x+x^{(1)}_{a})\Big\|_{\delta_{a}}
\|\bar{\xi}_{a}(y)\|_{\delta_{a}}\\
&\quad
+C\delta^{-1}_{a}\|U\|_{\delta_{a}}
\Big\|\tilde{\varphi}^{(1)}_{a}(\frac{\delta_{a}}{\sqrt{\textbf{p}_{0}}}x+x^{(1)}_{a})\Big\|_{\delta_{a}}
\Big\|\tilde{\varphi}^{(2)}_{a}(\frac{\delta_{a}}{\sqrt{\textbf{p}_{0}}}x+x^{(1)}_{a})\Big\|_{\delta_{a}}
\|\bar{\xi}_{a}(y)\|_{\delta_{a}}\\
&\quad
+C\delta^{-1}_{a}
\Big\|\tilde{\varphi}^{(1)}_{a}(\frac{\delta_{a}}{\sqrt{\textbf{p}_{0}}}x+x^{(1)}_{a})\Big\|^{2}_{\delta_{a}}
\|U\|_{\delta_{a}}
\|\bar{\xi}_{a}(y)\|_{\delta_{a}}
\\
&\quad
+C\delta^{-1}_{a}
\Big\|\tilde{\varphi}^{(1)}_{a}(\frac{\delta_{a}}{\sqrt{\textbf{p}_{0}}}x+x^{(1)}_{a})\Big\|^{2}_{\delta_{a}}
\Big\|\tilde{\varphi}^{(2)}_{a}(\frac{\delta_{a}}{\sqrt{\textbf{p}_{0}}}x+x^{(1)}_{a})\Big\|_{\delta_{a}}
\|\bar{\xi}_{a}(y)\|_{\delta_{a}}\\
&=O(\delta^{3}_{a})
+O(\delta^{7}_{a})+O(\delta^{11}_{a})
=O(\delta^{3}_{a}).
\end{split}
\end{equation*}

Just by the same argument as that of \eqref{uu-1-1}, we can also check that
\begin{equation}\label{uu-2}
\begin{split}
\int_{\R^{3}}\int_{\R^{3}}\frac{(u_a^{(2)})^{2}(x)u_a^{(1)}(y)
\xi_{a}(y)}{|x-y|}dxdy
=\delta_{a}^{5}\textbf{p}^{\frac{1}{2}}_{0}\int_{\R^{3}}\int_{\R^{3}}
\frac{U^{2}(x)U(y)\bar{\xi}_{a}(y)}{|x-y|}dxdy
+O(\delta^{6}_{a}).
\end{split}
\end{equation}
It follows from \eqref{uu-1-1} and
\eqref{uu-2} that \eqref{uu} holds.
\end{proof}

\vskip 0.3cm

\section{Acknowledgements}{Part of this work was done while Luo  was visiting the Mathematics Department of University of Rome "La Sapienza"
whose members he would like to thank for their warm hospitality. Guo was supported by NSFC grants (No.11771469). Luo and Wang were supported by the Fundamental
Research Funds for the Central Universities(No.KJ02072020-0319). Luo was supported by the China Scholarship Council and   NSFC grants (No.12171183, No.11831009).
Wang was supported by NSFC grants (No.12071169). Yang was supported by NSFC grants (No.11601194).}


\begin{thebibliography}{99}

\bibitem{Cao3}
D. Cao, H. Heinz,  {\em Uniqueness of positive multi-lump bound states of nonlinear Schr\"odinger equations,} Math. Z.{\bf 243} (2003), 599--642.

\bibitem{Cao1}
D. Cao, S. Li, P. Luo, {\em Uniqueness of positive bound states with multi-bump for nonlinear Schr\"odinger equations,} Calc. Var. Partial Differential Equations  {\bf 54} (2015), 4037--4063.


\bibitem{CNY}
D. Cao, E. Noussair, S. Yan, {\em Existence and uniqueness results on
single-peaked solutions of a semilinear problem,} Ann. Inst. H.
Poincar\'e Anal. Non Lin\'eaire {\bf 15} (1998), 73--111.

\bibitem{Chen-16}
G. Chen, {\em Nondegeneracy of ground states and multiple
semiclassical solutions of the Hartree equation for
general dimensions,} Results Math. {\bf 76}(2021), 1-34.

\bibitem{CGT}
S. Cingolani, M. Gallo, K. Tanaka,  {\em Symmetric ground states for doubly
nonlocal equations with mass constraint,} Symmetry, {\bf 13(7)}(2021), 1199.

\bibitem{Cingolani}
S. Cingolani, S. Secchi, M. Squassina,
  {\em Semi-classical limit for Schr\"{o}dinger equations with magnetic field and Hartree-type nonlinearities, }Proc. Roy. Soc. Edinburgh Sect. A  {\bf 140} (2010), 973--1009.

\bibitem{CT-2019}
S. Cingolani, K. Tanaka, {\em Semi-classical states for the nonlinear Choquard
equations: existence, multiplicity and concentration at a potential well,}
Rev. Mat. Iberoam. {\bf 35} (2019), 1885-1924.

\bibitem{DGL-18JMP}
Y. Deng, Y. Guo, L. Lu, {\em Threshold behavior and uniqueness
of ground states for mass critical inhomogeneous Schr\"odinger equations,} J. Math. Phys. {\bf 59} (2018), 21 pp.



\bibitem{Deng}
Y. Deng, C. Lin, S. Yan, {\em On the prescribed scalar curvature problem in $\R^N$, local uniqueness and periodicity,}  J. Math. Pures Appl. {\bf 104 } (2015), 1013--1044.

\bibitem{GW-20ANS}
B. Gheraibia, C. Wang, {\em Multi-peak positive solutions of a nonlinear Schr\"odinger-Newton type system,} Adv. Nonlinear Stud.{\bf 20} (2020), 53--75.

\bibitem{G}
L.  Glangetas,  {\em Uniqueness of positive solutions of a nonlinear elliptic equation involving the critical exponent,} Nonlinear Anal. {\bf 20} (1993), 571--603.

\bibitem{Grossi}
M. Grossi, {\em On the number of single-peak solutions of the nonlinear Schr\"odinger equations,}
Ann. Inst. H. Poincar\'e Anal. Non Lin\'eaire  {\bf 19} (2002), 261--280.

\bibitem{GLW}
Y. Guo, C. Lin, J. Wei, {\em Local uniqueness and refined spike profiles of ground states for two-dimensional attractive Bose-Einstein condensates,} SIAM J. Math. Anal. {\bf 49} (2017), 3671--3715.

\bibitem{GLW-18JMP}
Y. Guo, Y. Luo, Z-Q Wang,
 {\em Limit behavior of mass critical Hartree minimization problems with steep potential wells,} J. Math. Phys. {\bf 59} (2018),  061504, 19 pp.

\bibitem{GPY}
 Y. Guo, S. Peng, S. Yan, {\em Local uniqueness and periodicity induced by concentration,} Proc. Lond. Math. Soc. {\bf 114} (2017), 1005--1043.

\bibitem{JL}
L. Jeanjean, T.T. Le, {\em Multiple normalized solutions for a Sobolev critical
Schr\"odinger-Poisson-Slater equation,} J. Differential Equations {\bf 303} (2021),
277-325.

\bibitem{LPW-11JMP}
  G. Li, S. Peng, C. Wang, {\em Multi-bump solutions for the nonlinear Schr\"odinger-Poisson system,} J. Math. Phys. {\bf 52} (2011), 19 pp.

\bibitem{Li}
 G. Li, S. Peng,  S. Yan, {\em Infinitely many positive solutions for the nonlinear Schr\"odinger-Poisson system,} Commun. Contemp. Math.  {\bf 12 } (2010), 1069--1092.
 \bibitem{LY-14JMP}
 G. Li, H. Ye,   {\em The existence of positive solutions with prescribed $L^{2}$-norm for nonlinear Choquard equations,}
  J. Math. Phys. {\bf 55 }(2014), 121501, 19 pp.

\bibitem{LXZ-16-JMP}
S. Li, J.  Xiang, X. Zeng, {\em Ground states of nonlinear Choquard equations with multi-well potentials,  } J. Math. Phys. {\bf 57} (2016), 081515, 19 pp.

\bibitem{LZ-19ZAMP}
S. Li, X. Zhu, {\em Mass concentration and local uniqueness of ground states for $L^{2}$-subcritical nonlinear Schr\"odinger equations,} Z. Angew. Math. Phys. {\bf 70} (2019), 26 pp.

\bibitem{Lieb1}
E. Lieb, {\em Existence and uniqueness of the minimizing solution of Choquard's nonlinear equation,} Studies in Appl. Math. {\bf 57} (1976/77), 93--105.

\bibitem{Lieb}
 E. Lieb, M. Loss,{\em  Analysis, second ed., in: Graduate Studies in Mathematics,} vol. 14, American Mathematical Society, Providence, RI, 2001.

\bibitem{Lions}
 P. L. Lions,  {\em The Choquard equation and related questions,} Nonlinear Anal.  {\bf 4} (1980),  1063--1072.

\bibitem{Lions1}
P. L. Lions, {\em The concentration-compactness principle in the calculus of variations, The locally compact case II,}  Ann. Inst. H. Poincar\'e Anal. Non Lin\'eaire  {\bf 1} (1984), 223--283.

\bibitem{LPW-20CV}
P. Luo, S. Peng, C. Wang, {\em Uniqueness of positive solutions with concentration for the Schr\"odinger-Newton problem,}
 Calc. Var. Partial Differential Equations {\bf 59} (2020), Paper No. 60, 41 pp.

\bibitem{LPY-19}
P. Luo, S. Peng, J. Wei, S. Yan, {\em Excited states of Bose-Einstein condensates with degenerate attractive interactions,} Calc. Var. Partial Differential Equations {\bf 60} (2021), Paper No. 155, 33 pp.

 \bibitem{Menzala}
   G. P. Menzala, {\em On regular solutions of a nonlinear equation of Choquard's type,} Proc. Roy. Soc. Edinburgh Sect. A {\bf 86} (1980), 291--301.

 \bibitem{Moroz}
 V. Moroz, J. Van Schaftingen,{\em Ground states of nonlinear Choquard equations: existence, qualitative properties and decay asymptotics,} J. Funct. Anal. {\bf 265} (2013), 153--184.

 \bibitem{Moroz1}
 V. Moroz, J. Van Schaftingen,{\em Existence of ground states for a class of nonlinear Choquard equations,} Trans. Amer. Math. Soc.
 {\bf 367} (2015), 6557--6579  .

\bibitem{PPVV-2021}
B. Pellacci, A. Pistoia, G. Vaira, G. Verzini, {\em Normalized concentrating
solutions to nonlinear elliptic problems, } J. Differential Equations {\bf 275}
(2021), 882-919.

\bibitem{Penrose}
R. Penrose, {\em Quantum computation, entanglement and state reduction,} R. Soc. Lond. Philos. Trans. Ser. A Math. Phys. Eng. Sci.
{\bf 356} (1998),  1927--1939.

\bibitem{R-1949}
M. Riesz, {\em L'int\'egrale de Riemann-Liouville et le prodl\'eme de Cauchy,} Acta Math. {\bf 81} (1949)
1--233.

\bibitem{Secchi}
S. Secchi, {\em A note on Schr\"odinger-Newton systems with decaying electric potential, }Nonlinear Anal. {\bf 72} (2010), 3842--3856.

\bibitem{Tod}
P. Tod, I. M. Moroz, {\em An analytical approach to the Schr\"odinger-Newton equations,} Nonlinearity  {\bf 12} (1999), 201--216.

\bibitem{Vaira}
G. Vaira, {\em Existence of bound states for Schr\"odinger-Newton type systems,} Adv. Nonlinear Stud.  {\bf 13} (2013), 495--516.

\bibitem{WY-16DCDS}
C. Wang, J. Yang, {\em Positive solutions for a nonlinear Schr\"odinger-Poisson system, }Discrete Contin. Dyn. Syst. {\bf 38} (2018), 5461--5504.

\bibitem{Wei}
 J. Wei, M. Winter, {\em Strongly interacting bumps for the Schr\"odinger-Newton equations,} J. Math. Phys.  {\bf 50} (2009), 012905.


 \bibitem{Xiang}
 C. Xiang, {\em Uniqueness and nondegeneracy of ground states for Choquard equations in three dimensions,} Calc. Var. Partial Differential Equations  {\bf 55} (2016), 55--134.

\bibitem{Ye-16TMNA}
H. Ye, {\em Mass minimizers and concentration for nonlinear Choquard equations in $\R^{N}$,}
Topol. Methods Nonlinear Anal. {\bf 48} (2016), 393--417.

\end{thebibliography}
\end{document}